\documentclass{article}

\usepackage{graphicx,xcolor,textpos}
\usepackage{helvet}
\usepackage[margin = 1in]{geometry}
\usepackage[utf8]{inputenc}
\usepackage[english]{babel}
\usepackage[T1]{fontenc}
\usepackage{gensymb}
\usepackage{float}
\usepackage{caption}
\usepackage{subcaption}
\usepackage{amsmath, mathtools, mathrsfs}
\usepackage{amssymb}
\usepackage{amsthm}
\usepackage{enumitem}
\usepackage[makeroom]{cancel}
\usepackage{hyperref}
\usepackage{tikz}
\usepackage{tikz-cd}
\usepackage{blkarray}
\usepackage[framemethod = default]{mdframed}
\usepackage{multirow}
 \usetikzlibrary{calc}
\usepackage{todonotes}

\newtheorem{theorem}{Theorem}[section]
\theoremstyle{definition} 
\newtheorem{definition}[theorem]{Definition}
\theoremstyle{plain}
\newtheorem{prop}[theorem]{Proposition}
\theoremstyle{definition} 
\newtheorem{example}[theorem]{Example}
\theoremstyle{plain}
\newtheorem{lemma}[theorem]{Lemma}

\theoremstyle{definition}
\newtheorem{remark}[theorem]{Remark}
\theoremstyle{definition}
\newtheorem{question}[theorem]{Question}
\theoremstyle{plain}
\newtheorem{theoremintro}{Theorem}[section]

\DeclareMathOperator{\GL}{GL}

\DeclareMathOperator{\Aut}{Aut}

\DeclareMathOperator{\Id}{Id}

\DeclareMathOperator{\Gal}{Gal}

\DeclareMathOperator{\Perm}{Perm}
\DeclareMathOperator{\rd}{red}
\DeclareMathOperator{\spn}{span}
\DeclareMathOperator{\Hom}{Hom}

\DeclareMathOperator{\Sp}{span}
\DeclareMathOperator{\ab}{ab}

\newcommand{\Z}{\mathbb{Z}}
\newcommand{\R}{\mathbb{R}}
\newcommand{\C}{\mathbb{C}}
\newcommand{\Q}{\mathbb{Q}}
\newcommand{\Qbar}{\overline{\mathbb{Q}}}
\newcommand{\N}{\mathbb{N}}
\newcommand{\G}{\mathcal{G}}

\newcommand{\m}{\mathfrak{m}}
\newcommand{\g}{\mathfrak{g}}

\newcommand{\n}{\mathfrak{n}}

\newcommand{\asi}[1]{\prescript{\sigma}{}{#1}}

\title{Classification of $K$-forms in nilpotent Lie algebras\\ associated to graphs}

\author{Jonas Der\'e and Thomas Witdouck\thanks{KU Leuven Campus Kortrijk Kulak, Department of Mathematics. The second author was supported by a PhD fellowship of the Research Fund - Flanders (FWO).}}

\begin{document}
	\maketitle
	
	\begin{abstract}
	Given a simple undirected graph, one can construct from it a $c$-step nilpotent Lie algebra for every $c \geq 2$ and over any field $K$, in particular also over the real and complex numbers. These Lie algebras form an important class of examples in geometry and algebra, and it is interesting to link their properties to the defining graph. In this paper, we classify the isomorphism classes of $K$-forms in these real and complex Lie algebras for any subfield $K \subset \C$ from the structure of the graph. As an application, we show that the number of rational forms up to isomorphism is always one or infinite, with the former being true if and only if the group of graph automorphisms is generated by transpositions.
	\end{abstract}
	
	\section{Introduction}
	
	Given a field extension $K \subset L$ and Lie algebras $\mathfrak{n}^L$ and $\mathfrak{m}^K$ defined over $L$ and $K$ respectively, we say that $\mathfrak{m}^K$ is a $K$-form of $\mathfrak{n}^L$ if the tensor product $\mathfrak{m}^K \otimes_K L$, so by extending the scalars on $\mathfrak{m}^K$, is isomorphic to $\mathfrak{n}^L$ as Lie algebras. For a more detailed discussion of this notion, we refer to section \ref{sec:GaloisCohomology} below. Determining all $K$-forms up to $K$-isomorphism of a general Lie algebra over $L$ is a hard problem, which has been solved only for some specific fields and Lie algebras before. In this paper, we study this problem for the class of nilpotent Lie algebras associated to graphs, a family generalizing free $c$-step nilpotent Lie algebras and closed under taking direct sums. These Lie algebras have recently been studied by many different authors, including their automorphism group in \cite{dm21-1} and their geometric properties in \cite{ddm18-1,niko20-1,ovan22-1}.
	
	We first motivate the study of $K$-forms in (nilpotent) Lie algebras by relating it to several problems in algebra and geometry. As Grunewald and Segal discuss in \cite{gs84-1}, determining all rational forms of real nilpotent Lie algebras is part of the classification problem for finitely generated torsion-free nilpotent groups. This follows from the classical result of Mal'cev in \cite{malc49-2} that gives a one-to-one correspondence between rational nilpotent Lie algebras up to $\Q$-isomorphism and finitely generated torsion-free nilpotent groups up to abstract commensurability, i.e.~up to having an isomorphic subgroup of finite index. One of the projects suggested by Grunewald and Segal, see \cite[3.5.]{gs84-1}, is to characterize rational forms for certain interesting classes of complex nilpotent Lie algebras, as a first step to understand the finitely generated torsion-free nilpotent groups corresponding to them. Our results can be seen as an answer to this question in the large class of nilpotent Lie algebras associated to graphs.
	
	The classification of rational forms in real nilpotent Lie algebras is also of importance in certain geometrical problems on nilmanifolds, i.e.~the quotients of a simply connected nilpotent Lie groups by a cocompact lattice. Indeed, such a lattice is a finitely generated torsion-free nilpotent group and thus gives a unique rational Lie algebra which in addition is a rational form of the real nilpotent Lie algebra corresponding to the Lie group. Geometrical problems on the nilmanifold often only depend on the isomorphism type of the rational form. Examples of such problems are the existence of Anosov diffeomorphisms and expanding maps, important types of dynamical systems that conjecturally only exist on manifolds finitely covered by nilmanifolds, see  \cite{dd14-1,dv11-1,smal67-1} for more details. In a subsequent paper, we will apply the results of this paper to find a full characterization of Anosov diffeomorphisms on nilmanifolds modeled on Lie groups associated to graphs, following previous work in \cite{dm05-1,dm21-1}.
	
	As real Lie algebras are exactly the ones corresponding to Lie groups, finding real forms of a complex Lie algebra is important for studying Lie groups. For example, the classification of semisimple Lie groups, or hence of semisimple real Lie algebras heavily depends on the study of real forms and the relation to the underlying real Lie algebra of complex Lie algebras, see e.g.~\cite{helg01-1}. Furthermore, having a real form determines the structure of the complex conjugate Lie algebras as described in \cite{dere19-1} and hence from \cite{dlv12-1} the number of distinct almost-complex structures on nilmanifolds having a quasi-K\"ahler Chern-flat metric in the $2$-step nilpotent case. Note that it is well-known that any complex Lie algebra can have at most finitely many real forms up to isomorphism, see \cite[Remark 4.3.]{dlv12-1}.
	
	So far, the existence of $K$-forms has only been studied for some specific classes of Lie algebras, mostly in low dimensions. For example, it is an easy exercise to show that free nilpotent Lie algebras over any field $L$ have exactly one $K$-form for every $K \subset L$. In \cite{seme02-1} the author considers several classes of $6$-dimensional real nilpotent Lie algebras and shows via direct computations that they have either $1$ or infinitely many rational forms. This paper also gives a method to construct rational forms in the direct sum of two free nilpotent Lie algebras. Later, Lauret computes rational forms of certain real nilpotent Lie algebras of dimension $8$ in \cite{laur08-1} using the Pfaffian form, including a description which ones correspond to a nilmanifold admitting an Anosov diffeomorphism. Note that all of the Lie algebras in \cite{seme02-1}, including the direct sum of free Lie algebras, and some of the Lie algebras in \cite{laur08-1} fall in the class of Lie algebras associated to a graph. Our results work for any subfield $K \subset \C$, but below we present them in the special case of $K = \Q$ and $K= \R$, as these are the most important for the applications mentioned above.

	Given a graph $\G$, one can define an equivalence relation on the vertices by saying two vertices are equivalent if and only if their transposition defines a graph automorphism. This gives rise to a quotient graph $\overline{\G}$, for which we give the exact definitions in Section \ref{sec:quotientGraph}. At the beginning of section \ref{sec:formsNilpLieAlgGraphs}, we recall how one can construct for any integer $c > 1$ a $c$-step nilpotent Lie algebra $\n_{\G,c}^K$ from the graph $\G$ over a field $K \subset \C$. If $\Gal(L/K)$ is a Galois extension of subfields of $\C$, then the Galois group $\Gal(L/K)$ naturally acts on $\n_{\G, c}^L$ by semi-linear maps which fix the vertices. This action will be written as $\asi{v}$ for any $\sigma \in \Gal(L/\Q)$ and $v \in \n_{\G, c}^L$. In Section \ref{sec:classKformsL}, we define for any field $K \subset \C$ a group morphism
	\[i:\Aut(\overline{\G}) \to \Aut(\n_{\G,c}^K),\]
	which is in fact not natural but easy to construct after one chooses an order on the vertices. From now on, endow $\Aut(\overline{\G})$ with the discrete topology and $\Gal(L/K)$ with the Krull topology (which is the discrete topology if $\Gal(L/K)$ is finite). For any continuous morphism $\rho:\Gal(L/K) \to \Aut(\overline{\G}): \sigma \mapsto \rho_\sigma$, we define the subset
	\[ \n_{\rho, c}^K = \{ v \in \n^L_{\G, c} \mid \forall \sigma \in \Gal(L/K): i(\rho_\sigma)(\asi{v}) = v \} \]
	which is a $K$-Lie algebra if one restricts the addition, $K$-scalar multiplication and Lie bracket of $\n_{\G, c}^L$ to it. As mentioned in section \ref{sec:GaloisCohomologyDefinitions}, a classical result states that $\n_{\rho, c}^K$ is a $K$-form of $\n_{\G, c}^L$. 
	
	As a main result of this paper, it is shown that up to $\Q$-isomorphism, the forms defined above are all the rational forms of $\n_{\G, c}^\C$. All fields in the theorem are considered as subfields of $\C$.
	
	\begin{theoremintro}[Rational forms]
		\label{thm:injectiveVersionClassificationRationalForms}
		Let $\G$ be a simple undirected graph and $\n_{\G, c}^\C$ the associated $c$-step nilpotent complex Lie algebra. Any rational form of $\n_{\G,c}^\C$ is $\Q$-isomorphic to $\n_{\rho, c}^\Q$ for some finite degree Galois extension $L/\Q$ and an injective group morphism $\rho:\Gal(L/\Q) \to \Aut(\overline{\G})$. If $K/\Q$ is another finite degree Galois extension with an injective group morphism $\eta:\Gal(K/\Q) \to \Aut(\overline{\G})$, then $\n_{\rho, c}^\Q$ and $\n_{\eta, c}^\Q$ are $\Q$-isomorphic if and only if $L = K$ and there exists a $\varphi \in \Aut(\overline{\G})$ such that $\varphi \rho(\sigma) \varphi^{-1} = \eta(\sigma)$ for all $\sigma \in \Gal(L/\Q)$. The Lie algebra $\n_{\rho, c}^\Q$ is a rational form of $\n_{\G, c}^\R$ if and only if $L \subset \R$.
	\end{theoremintro}
	
	\noindent As we show in Theorem \ref{thm:numberofrational}, the above result also implies that $\n_{\G, c}^\C$ or $\n_{\G, c}^\R$ has either one or infinitely many rational forms. The former is true if and only if the group $\Aut(\overline{\G})$ is trivial, which is the case if and only if the group $\Aut(\G)$ is generated by transpositions.
	
	Note that one important restriction to apply Theorem \ref{thm:injectiveVersionClassificationRationalForms} is that in general it is not known which finite groups can occur as the Galois group $\Gal(K/ \Q)$ of a finite degree Galois extension $\Q \subset K$, a problem known as the inverse Galois problem. In the special case where $\Aut(\overline{\G})$ has order 2, all injective group morphisms are of the form $\Gal(\Q(\sqrt{d})/\Q) \to \Aut(\overline{\G})$ with $d \in \Z \setminus \{0\}$ and hence it is possible to explicitly write down the structural constants of all rational forms. We illustrate this approach in example \ref{ex:twoCopiesHeisenberg}.
	
	For the real forms of the complex Lie algebra, on the other hand, this issue does not occur as the Galois group $\Gal(\C / \R)$ is known. Let $\tau \in \Gal(\C/\R)$ denote the complex conjugation automorphism, then for any involution $\varphi \in \Aut(\overline{\G})$ we can define the real form
	\[ \n_{\varphi, c}^K := \n_{\rho, c}^K \quad \quad \text{with} \quad \rho:\Gal(\C/\R) \to \Aut(\overline{\G}): \tau \mapsto \varphi.\]

	\begin{theoremintro}[Real forms]
		\label{thm:ClassificationOfRealForms}
		Let $\G$ be a simple undirected graph and $\n_{\G, c}^\C$ the associated $c$-step nilpotent complex Lie algebra. Every real form of $\n_{\G, c}^\C$ is $\R$-isomorphic to $\n_{\varphi,c}^\R$ for some involution $\varphi \in \Aut(\overline{\G})$. If $\phi \in \Aut(\overline{\G})$ is another involution then the real forms $\n_{\varphi, c}^\R$ and $\n_{\phi, c}^\R$ are $\R$-isomorphic if and only if $\varphi$ and $\phi$ are conjugate in $\Aut(\overline{\G})$.
	\end{theoremintro}

	\noindent Theorem \ref{thm:ClassificationOfRealForms} thus gives a method to compute the exact number of real forms in $\n_{\G,c}^\C$ from only the graph $\G$. As an application we construct several examples of complex Lie algebras with exactly $k$ different real forms for every $k \in \N$.
	
	A Lie algebra is called decomposable if it is isomorphic to the direct sum of two non-trivial ideals. In general, a $K$-form of a decomposable Lie algebra is not always decomposable itself. Any automorphism of the quotient graph induces a permutation on the set of connected components $\mathcal{C}(\G)$ of the graph $\G$. This induces a morphism $\chi:\Aut(\overline{\G}) \to \Perm(\mathcal{C}(\G))$, for which the exact definition is given in section \ref{sec:indecomposableForms}. The following theorem characterizes the indecomposable $K$-forms of a Lie algebra associated to a graph.
	
	\begin{theoremintro}[Indecomposable $K$-forms]
		\label{thm:indecomposableForms}
		Let $\G = (S, E)$ be a graph, $L/K$ a Galois extension of subfields of $\C$ and $\rho:\Gal(L/K) \to \Aut(\overline{\G})$ a continuous morphism. The Lie algebra $\n_{\rho, c}^K$ is indecomposable if and only if the $\chi \circ \rho$-action on the set of connected components $\mathcal{C}(\G)$ is transitive.
	\end{theoremintro}
	
	\noindent We start in section \ref{sec:formsNilpLieAlgGraphs} by recalling the results from \cite{dm21-1} that describe the automorphim group of 2-step nilpotent Lie algebras associated to graphs and show how this result can be extended to the general $c$-step nilpotent case. In order to prove Theorems \ref{thm:injectiveVersionClassificationRationalForms} and \ref{thm:ClassificationOfRealForms}, we introduce the basic notions of Galois cohomology in Section \ref{sec:GaloisCohomology}, including some preliminary results on computations for certain semi-direct products.  In section \ref{sec:classKformsL} the preceding results are combined to calculated the Galois cohomology of the automorphism group of $c$-step nilpotent Lie algebras associated to graphs. Using the principle of Galois descent, Theorems \ref{thm:injectiveVersionClassificationRationalForms} and \ref{thm:ClassificationOfRealForms} are proven in sections \ref{sec:rationalForms} and \ref{sec:realForms}, respectively. At last, we give some applications of our main results in section \ref{sec:indecomposableForms}, including the proof of Theorem \ref{thm:indecomposableForms}.
	
	\section{Lie algebras associated to graphs and their automorphisms}
	\label{sec:formsNilpLieAlgGraphs}
	
	In this section, all fields will be considered as subfields of $\C$. We start by recalling a construction which associates to a graph $\G$ a Lie algebra, and for every $c > 1$ a $c$-step nilpotent Lie algebra. These Lie algebras are also referred to as \textit{free ($c$-step nilpotent) partially commutative Lie algebras}. Let us start by defining what we mean by a graph and its automorphisms.
	
	\begin{definition}
		A \textit{simple undirected graph} is a pair $\G = (S,E)$ with $S$ a finite set and $E$ a subset of $\{ \{\alpha, \beta \} \subset S \mid \alpha \neq \beta \}$. A bijection $\theta:S \to S$ is called an automorphism of the graph $\G$ if for any edge $\{\alpha, \beta\} \in E$ we have $\{\theta(\alpha), \theta(\beta)\} \in E$. The set of all automorphisms of $\G$ forms a group under composition and is denoted by $\Aut(\G)$.
	\end{definition}
	
	Take any subfield $K \subset \C$ and let $\mathfrak{f}^K(S)$ denote the free Lie algebra on $S$ over the field $K$. We write $I_\G$ for the ideal of $\mathfrak{f}^K(S)$ generated by the set of brackets $\{ [\alpha, \beta] \mid \{\alpha, \beta\} \notin E \}$. We define the \textit{Lie algebra $\g_{\G}^K$ associated to the graph $\G$ over the field $K$} by the quotient
	\begin{equation}
		\label{eq:defintionInfLieAlgGraph}
		\g_{\G}^K := \mathfrak{f}^K(S)/I_\G.
	\end{equation}
	For any Lie algebra $\g$, we let $\gamma_{i}(\g)$ denote the $i$-th ideal of the lower central series of $\g$ which is defined inductively by $\gamma_1(\g) = \g$ and $\gamma_{i+1}(\g) = [\g, \gamma_{i}(\g)]$. We can make $\g_{\G}^K$ $c$-step nilpotent by taking the quotient by $\gamma_{c + 1}(\g_\G^K)$ obtaining the \textit{$c$-step nilpotent Lie algebra associated to the graph $\G$}:
	\[ \n_{\G,c}^K := \g_{\G}^K/\gamma_{c + 1}(\g_\G^K). \]
	The special case of $c=2$ was introduced in \cite{dm05-1} and has ever since provided a rich class of $2$-step nilpotent Lie algebras. Later these Lie algebras were considered for higher nilpotency class, for example in \cite{main06-1}.
	
	With abuse of notation we will write $V := \spn_K(S)$ for the vector space spanned by the vertices in any of the Lie algebras $\mathfrak{f}^K(S)$, $\g_{\G}^K$ and $\n^K_{\G,c}$. Recall that the free Lie algebra has a natural positive grading. If we define the vector subspaces $V^i$ inductively via $V^1 = V$ and $V^{i+1} = [V, V^{i}]$, this grading is given by
	\[\mathfrak{f}_c(S) = \bigoplus_{i = 1}^c V^i\]
	where $[V^i, V^j] = V^{i+j}$. If we write $W_\G = \spn(\{ [\alpha, \beta] \mid \alpha, \beta \in S, \{\alpha, \beta\} \notin E \}) \subset \mathfrak{f}^K(S)$, then the ideal $I_\G$ can be written as
	\[ I_\G = W_\G^1 + W_\G^2 + W_\G^3 + \ldots \]
	where $W_\G^i$ is defined inductively by $W_\G^1 = W_\G$ and $W_\G^{i + 1} = [V, W_\G^{i}]$. Note that we have $W_\G^{i} \subset V^{i+1}$ and thus we find that $\g_\G^K$ and $\n_{\G,c}^K$ have a positive grading given by
	\begin{equation}
	\label{eq:gradingPartComLieAlg}
	\g_{\G}^K = V \oplus \bigoplus_{i = 1}^\infty V^{i+1}/W_\G^{i}, \quad \quad \quad \n_{\G, c}^K = V \oplus \bigoplus_{i = 1}^{c-1} V^{i+1}/W_\G^{i}.
	\end{equation}
	
	\subsection{Graded automorphisms of $\g_{\G}^K$ and $\n_{\G, c}^K$}
	\label{sec:gradedAutomorphisms}
	
	Now that we know each of the Lie algebras $\mathfrak{f}^K(S)$, $\g_{\G}^K$ and $\n_{\G,c}^K$ has a positive grading, we can consider the automorphisms of these Lie algebras that preserve this grading, i.e. they map each summand of the grading onto itself. As it turns out, this condition is equivalent with the weaker assumption that the first summand $V$ is mapped onto itself. Let us write the subgroups of graded automorphisms of $\g_{\G}^K$ and $\n_{\G, c}^K$ as
	\begin{align*}
	T_{\G} &= \{ f \in \Aut(\g^K_{\G}) \mid f(V) = V \} \\
	T_{\G, c} &= \{ f \in \Aut(\n^K_{\G,c}) \mid f(V) = V \}.
	\end{align*}
	For any of the Lie algebras $\m = \mathfrak{f}^K(S), \, \g_{G}^K, \, \n_{\G,c}^K$, let $p$ denote the canonical projection morphism
	\begin{align}
	\label{eq:morphismp}
	p:\Aut(\m) \to \GL(V)
	\end{align}
	which takes an automorphism to its induced map on the abelianization $\m/[\m,\m]$ which can be identified with $V$ via the isomorphism $V \to \m/[\m,\m]: v \mapsto v + [\m, \m]$.
	
	\begin{prop}
		For any $c > 1$, the groups $p(T_{\G})$, $p(T_{\G, c})$, $p(\Aut(\g_{\G}^K))$ and $p(\Aut(\n_{\G, c}^K))$ are equal.
	\end{prop}

	\begin{proof}
		Note that any automorphism of $\g_{G}^K$ must preserve $\gamma_{c + 1}$ and thus must induce an automorphism on $\n_{\G,c}^K$. This gives us a map $h:\Aut(\g_{G}^K) \to \Aut(\n_{\G, c}^K)$. As is immediate from the definition, $h$ maps $T_\G$ into $T_{\G,c}$. We then get a commutative diagram
		\[
		\begin{tikzcd}
		T_{\G} \arrow[d, "h|_{T_{\G}}"] \arrow[r, hookrightarrow] & \Aut(\g_{\G}^K) \arrow[d, "h"] \arrow[r, "p"] & \GL(V) \arrow[d, "\Id"] \\
		T_{\G,c} \arrow[r, hookrightarrow] & \Aut(\n_{\G,c}^K) \arrow[r, "p"] & \GL(V).
		\end{tikzcd}
		\]
		We thus get the inclusions $p(T_\G) \subset p(\Aut(\g_{\G}^K)) \subset p(\Aut(\n_{\G, c}^K))$ and $p(T_\G) \subset p(T_{\G,c}) \subset p(\Aut(\n_{\G, c}^K))$. It thus suffices to prove the inclusion $p(\Aut(\n_{\G, c}^K)) \subset T_{\G}$.
		
		Take an arbitrary automorphism $f \in\Aut(\n_{\G,c}^K)$. By the universal property of the free Lie algebra $\mathfrak{f}(S)$, there exists a unique automorphism $g \in p(T_\G)$ with $g|_V = p(f)$. If we let $\pi$ denote the quotient morphism $\pi:\mathfrak{f}^K(S) \to \n_{\G,c}^K$, then we get for any $v \in V \subset \mathfrak{f}^K(S)$ that $(\pi \circ g)(v) - (f \circ \pi)(v) \in \gamma_2(\n_{\G, c}^k)$. As a consequence we also get for any $v \in V^2 = [V,V]$ that $(\pi \circ g)(v) - (f \circ \pi)(v) \in \gamma_3(\n_{\G, c}^k)$. Thus if $w \in W_\G \subset V^2$, we get that $\pi(w) = 0$ and as a consequence that $\pi(g(w)) \in \gamma_3(\n_{\G, c}^k)$. But since $\pi$ maps $[V, V]$ into $[V,V]/W_\G$, we must have $\pi(g(w)) = 0$ and thus that $g(w) \in W_\G$. We conclude that $g(W_\G) \subset W_\G$. 
		
		We can follow the same argument for $f^{-1}$, giving a graded automorphism $\tilde{g} \in \Aut(\mathfrak{f}^K(S))$ with $\tilde{g}|_V = p(f^{-1})$ and $\tilde{g}(W_\G) \subset W_\G$. Since $\tilde{g}|_V = p(f^{-1}) = p(f)^{-1} = g|_V^{-1}$ and $g$ is a graded automorphism of $\mathfrak{f}^K(S)$, we find that $\tilde{g} = g^{-1}$. Similarly as before, we find that $g^{-1}(W_\G) \subset W_\G$, leading to the equality $g(W_\G) = W_\G$. Inductively we then have that
		\[g(W_\G^{i+1}) = g([V, W_\G^i]) = [g(V), g(W_\G^i)] = [V, W_\G^i] = W_\G^{i+1}.\]
		This proves that $g(I_\G) = I_\G$ and that $g$ induces a graded automorphism $\overline{g}: \g^K_{\G} \to \g^K_{\G}$. By construction $p(\overline{g}) = p(f)$. Since $f$ was chosen arbitrarily in $\Aut(\n^K_{\G,c})$, this shows that $p(\Aut(\n_{\G, c}^K)) \subset p(T_\G)$ which concludes the proof.
	\end{proof}
	
	As a consequence of the above lemma, we can define a subgroup $G \leq \GL(V)$ as
	\[ G := p(T_\G) = p(T_{\G,c}) = p(\Aut(\g_{\G}^K)) = p(\Aut(\n_{\G,c}^K)).\]
	Note that when we restrict $p$ to $T_\G$ or $T_{\G,c}$ we actually get isomorphisms of groups
	\begin{equation}
	\label{eq:definitionptilde}
	\tilde{p}:T_\G \stackrel{\cong}{\longrightarrow} G: f \mapsto p(f), \quad \quad \quad \tilde{p}:T_{\G,c} \stackrel{\cong}{\longrightarrow} G: f \mapsto p(f).
	\end{equation}
	In \cite{dm21-1} the group $G \leq \GL(V)$ was completely determined in terms of the graph $\G$ (where the proof was done for two-step Lie algebras). Let us first recall the necessary definitions in order to state this structure theorem for $G$.
	
	Let $\G = (S, E)$ be a simple undirected graph. For any vertex $\alpha \in S$, we define the \textit{open and closed neighbourhoods of $\alpha$} by
	\begin{equation}
		\label{eq:openAndClosedNeigh}
		\Omega'(\alpha) = \{ \beta \in S \mid \{ \alpha,\beta \} \in E  \} \quad \text{and} \quad \Omega(\alpha) = \Omega'(\alpha) \cup \{\alpha\},
	\end{equation}
	respectively. This allows us to define a relation $\prec$ on the vertices by $\alpha \prec \beta \Leftrightarrow \Omega'(\alpha) \subset \Omega(\beta)$. An equivalence relation $\sim$ is then defined as $\alpha \sim \beta \Leftrightarrow \alpha \prec \beta \wedge \alpha \succ \beta$. Note that $\alpha \sim \beta$ if and only if the transposition of $\alpha$ with $\beta$ defines an automorphism of $\G$. The equivalence classes of $\sim$ are called the \textit{coherent components of $\G$} and are denoted by $\Lambda := S/\sim$. For $\alpha,\beta \in S$, let $E_{\alpha \beta} \in \text{End}(V)$ denote the linear map which satisfies
	\begin{equation*}
	E_{\alpha \beta}(\gamma) = 
	\begin{cases}
	\alpha \quad \text{if } \gamma = \beta\\
	0 \quad \text{else}
	\end{cases}
	\end{equation*}
	for all $\gamma \in S$. Then define the subgroup $M \leq \GL(V)$ by
	\[ M := \left\langle I_V + tE_{\alpha \beta} \,\, \Big| \,\, t \in K, \, \alpha \prec \beta, \alpha \nsucc \beta \right\rangle,  \]
	where $I_V$ is the identity on $V$. For any set $X$, we denote with $\Perm(X)$ the group of permutations of $X$. We can then define a group morphism $P: \Perm(S) \to \GL(V):\theta \mapsto P_\theta$ where $P_\theta$ is defined by $P_\theta(\alpha) = \theta(\alpha)$ for any $\alpha \in S$. At last, for a coherent component $\lambda \in \Lambda$, we will write $V_\lambda$ for the vector subspace of $V$ spanned by the vertices in $\lambda$. We can then view $\GL(V_\lambda)$ as a subgroup of $\GL(V)$ by identifying a linear map $A\in \GL(V_\lambda)$ with the linear map in $\GL(V)$ which maps $\alpha$ to itself whenever $\alpha \in S \setminus \lambda$ and maps $\alpha$ to $A(\alpha)$ whenever $\alpha \in \lambda$. 
	
	\begin{theorem}[Deré, Mainkar \cite{dm21-1}]
		\label{thm:structureThmG}
		The group $G$ is a linear algebraic subgroup of $\GL(V)$ (with respect to the basis $S$) and is given by
		\[ G = M \cdot \left( \prod_{\lambda \in \Lambda} \GL(V_\lambda) \right) \cdot P(\Aut(\G))\]
		where $M$ is equal to the unipotent radical of $G$.
	\end{theorem}
	
	\subsection{The quotient graph and a splitting morphism $r:\Aut(\overline{\G}) \to \Aut(\G)$}
	\label{sec:quotientGraph}
	
	The concept of a quotient graph on the coherent components was introduced in \cite{maplsa18-1} and is based on the following observation, also mentioned in \cite{dm05-1}.
	
	\begin{lemma}
		Let $\G = (S, E)$ be a graph and $\lambda, \mu \in \Lambda$ coherent components. If there exist $\alpha \in \lambda$ and $\beta \in \mu$ such that $\{\alpha, \beta\} \in E$, then it holds that for any $\alpha' \in \lambda$ and $\beta' \in \mu$, $\{\alpha', \beta'\} \in E$. This holds in particular when $\lambda = \mu$.
	\end{lemma}
	
	We can thus define a new graph which we will call the quotient graph, which has as its vertices the coherent components of the old graph. First let us define the category where this new object lies in.
	
	\begin{definition}
		A \textit{simple undirected vertex-weighted graph with loops} is a triple $(S,E,\Phi)$ where $S$ is a set, $E$ is a subset of $\{ \{s,t\} \mid s,t \in S \}$ and $\Phi:S \to \R$ is a map to $\R$. Its automorphism group is defined as
		\[ \Aut(S,E,\Phi) = \{ \varphi \in \text{Perm}(S) \mid e \in E \Leftrightarrow \varphi(e) \in E,\,\, \Phi \circ \varphi = \Phi \}. \]
	\end{definition}
	
	\noindent Let $\G = (S, E)$ be a simple undirected graph and let $\Lambda$ denote its set of coherent components, then associated to $\G$ is a simple undirected vertex-weighted graph with loops $\overline{\G}$ which we will call the \textit{quotient graph of $\G$}. It is defined as $\overline{\G} = (\Lambda, \overline{E}, \Phi)$ with \[ \overline{E} := \{ \{ \lambda, \mu \} \mid \exists \alpha \in \lambda, \exists \beta \in \mu: \{\alpha,\beta\} \in E \} \]
	and $\Phi: \Lambda \to \R: \lambda \mapsto |\lambda|$.
	\begin{example}
		We will make a visual representation of the quotient graph by simply putting the values of the weights near every vertex and drawing a loop at those $\lambda$ for which $\{\lambda\} \in \overline{E}$. Below we have drawn a concrete example.
		
		
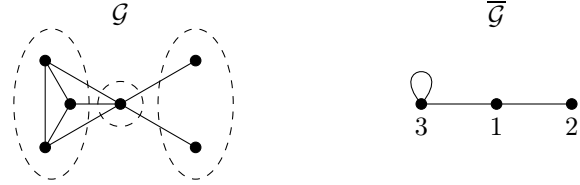
\begin{figure}[H]
			\centering
			\begin{tikzpicture}[every loop/.style={}]
			\draw (0, 0.5774) -- (1,0);
			\draw (0,-0.5774) -- (1,0);
			\draw (0, 0.5774) -- (1/3,0);
			\draw (0,-0.5774) -- (1/3,0);
			\draw (1/3,0) -- (1,0);
			\draw (0, 0.5774) -- (0,-0.5774);
			\draw (2, 0.5774) -- (1,0);
			\draw (2,-0.5774) -- (1,0);
			\filldraw [black] (1,0) circle (2pt);
			\filldraw [black] (0,0.5774) circle (2pt);
			\filldraw [black] (0,-0.5774) circle (2pt);
			\filldraw [black] (1/3,0) circle (2pt);
			\filldraw [black] (2,0.5774) circle (2pt);
			\filldraw [black] (2,-0.5774) circle (2pt);
			
			\node at (1, 1.2) {$\G$};
			
			\draw (5,0) -- (6,0);
			\draw (6,0) -- (7,0);
			\filldraw [black] (6,0) circle (2pt) node[inner sep=5pt, anchor = north] {1};
			\filldraw [black] (7,0) circle (2pt) node[inner sep=5pt, anchor = north] {2};
			\filldraw [black] (5,0) circle (2pt) node[inner sep=5pt, anchor = north] {3};
			\draw[cm={1.5 ,0 ,0 ,1.5 ,(5,0)}] (0,0)  to[in=50,out=130, loop] (0,0);
			
			\node at (6, 1.2) {$\overline{\G}$};
			
			\draw [black, dashed] (1,0) circle (0.3);
			\draw [black, dashed, cm={0.5 ,0 ,0 ,1 ,(0,0)}] (1/6,0) circle (1);
			\draw [black, dashed, cm={0.5 ,0 ,0 ,1 ,(2,0)}] (0,0) circle (1);
			
			\end{tikzpicture}.
		\end{figure}
	\end{example}
	
	Since any graph automorphism $\theta \in \Aut(\G)$ preserves the relation $\prec$, it also preserves the equivalence relation $\sim$. Therefore $\theta$ induces a permutation $\overline{\theta}$ on the set of coherent components $\Lambda$. In fact the induced permutation $\overline{\theta}$ is an automorphism of the quotient graph $\overline{\G}$. By consequence we get a morphism $\Aut(\G) \to \Aut(\overline{\G}): \theta \mapsto \overline{\theta}$ which fits in the short exact sequence:
	\begin{equation}
	\label{eq:splitSESgraphpermutations}
	\begin{tikzcd}
	1 \arrow[r] &\prod_{\lambda \in \Lambda} \text{Perm}(\lambda) \arrow[r, hookrightarrow] &\Aut(\G) \arrow[r, shift left] &\Aut(\overline{\G}) \arrow[l, shift left, dashed, "\exists r"] \arrow[r] &1
	\end{tikzcd}
	\end{equation}
	where $\Perm(X)$ denotes the group of permutations on the set $X$. In fact, this sequence is right-split, i.e. there exists a morphism $r:\Aut(\overline{\G}) \to \Aut(\G)$ such that $\overline{r(\varphi)} = \varphi$ for any $\varphi \in \Aut(\G)$. Such a morphism $r$ is not unique, but let us show how to construct one. First choose an ordering of the vertices inside each coherent component $\lambda = \left\{ \alpha_{\lambda,1}, \alpha_{\lambda,2}, \ldots, \alpha_{\lambda,\Phi(\lambda)} \right\}$. Then for $\varphi \in \Aut(\overline{\G})$, define $r(\varphi) \in \Aut(\G)$ by $r(\varphi)(\alpha_{\lambda,i}) = \alpha_{\varphi(\lambda), i}$ for all $\lambda \in \Lambda$ and $1 \leq i \leq \Phi(\lambda)$. One can check that $r$ is well-defined and satisfies $\overline{r(\varphi)} = \varphi$ for any $\varphi \in \Aut(\overline{\G})$. Moreover, if $\varphi \in \Aut(\overline{\G})$ has a fixed point $\varphi(\lambda) = \lambda$, then $r(\varphi)|_\lambda = \Id_\lambda$. This is not necessarily true for any choice of $r$ that makes the short exact sequence right-split, but is true for the one we constructed above. Let us for the remainder of this paper always fix such a morphism $r$ which does enjoy this additional property.
	
	Note that this morphism $r$ allows us to write $\Aut(\G)$ as a semi direct product
	\begin{equation} 
	\label{eq:DecompositionAutGraph}
	\Aut(\G) \cong \left(\prod_{\lambda \in \Lambda} \Perm(\lambda) \right) \rtimes_r \Aut(\overline{\G}).
	\end{equation}
	Analogous to the morphism $P:\Aut(\G) \to G$ which maps an automorphism of $\G$ to its corresponding permutation matrix in the basis of vertices, we can now also define a morphism $\overline{P}: \Aut(\overline{\G}) \to G$ by the commutative diagram:
	\begin{equation}
	\begin{tikzcd}
	\Aut(\G) \arrow[r, "P"] & G\\
	\Aut(\overline{\G}) \arrow[u, "r"'] \arrow[ru, "\overline{P}"']. &
	\end{tikzcd}
	\end{equation}
	Using the above we find that
	\[P(\Aut(\G)) = P\left( \left(\prod_{\lambda \in \Lambda} \Perm(\lambda)\right) \cdot  r(\Aut(\overline{\G}))  \right) = \left(\prod_{\lambda \in \Lambda} P(\Perm(\lambda))\right) \cdot \overline{P}(\Aut(\overline{\G})). \]
	Together with Theorem \ref{thm:structureThmG} we thus find that
	\[ G = M \cdot \left( \prod_{\lambda \in \Lambda} \GL(V_\lambda) \right) \cdot \overline{P}(\Aut(\overline{\G})) \]
	where the subgroups $M$, $\left( \prod_{\lambda \in \Lambda} \GL(V_\lambda) \right)$ and $\overline{P}(\Aut(\overline{\G}))$ have pairwise trivial intersections by \cite[Lemma 3.10.]{dm21-1}. Moreover, we have that $\overline{P}(\Aut(\overline{\G}))$ normalizes both $M$ and $\prod_{\lambda \in \Lambda} \GL(V_\lambda)$ and that $\prod_{\lambda \in \Lambda} \GL(V_\lambda)$ normalizes $M$. We thus have a semi-direct product decomposition of $G$:
	\[ G \cong \left( M \rtimes \left(\prod_{\lambda \in \Lambda} \GL(V_\lambda) \right) \right) \rtimes \Aut(\overline{\G}) . \]
	Note that $G^0 := M \cdot \left( \prod_{\lambda \in \Lambda} \GL(V_\lambda) \right)$ is the Zariski-connected component of the identity in $G$ and that by taking the quotient by it, we get a morphism
	\begin{equation}
	\label{eq:definitionq}
	q: G \to G/G^0 \cong \Aut(\overline{\G}).
	\end{equation}
	
	\section{Galois cohomology}
	\label{sec:GaloisCohomology}
	
	In this section we introduce the tools from Galois cohomology that we will use throughout this paper. We first give an overview of the main definitions in relation to the study of $K$-forms of Lie algebras. Next, we present a technical result on the Galois cohomology of a class of semi-direct products, including wreath products, which will be fruitful for the computations in the next section. This is the only section in the paper where we do not require the fields to be subfields of $\C$ in general.
	
	\subsection{Definitions and relation to $K$-forms}
	\label{sec:GaloisCohomologyDefinitions}
	
	First we recall some notions from Galois theory of possibly infinite degree field extensions. If $L/K$ is an extension of fields i.e.~$K \subset L$, then we denote with $\Aut(L/K)$ the group of field automorphisms of $L$ that fix every element of $K$. For any subgroup $H \leq \Aut(L/K)$, we write $L^H$ for the field fixed by $H$, i.e. $L^H = \{ l \in L \mid \forall \sigma \in H: \sigma(l) = l  \}$.
	
	\begin{definition}
		A field extension $L/K$ is called \textit{Galois} if it is algebraic over $K$ and if $L^{\Aut(L/K)} = K$. In this case we write $\Gal(L/K)$ for $\Aut(L/K)$.
	\end{definition}
	
	Note that the field extension in the definition above is not required to be of finite degree. We can put a topology on $\Gal(L/K)$ turning it into a topological group and more specific into a \textit{profinite group}. This topology is called the \textit{Krull topology} and a basis of opens for it is given by
	\[ \left\{  \sigma \Gal(L/N) \,\, \Big| \,\, \sigma \in \Gal(L/K) \text{ and } N/K \text{ is an intermediate field extension of finite degree} \right\}. \]
	In the special case when $L/K$ is a finite degree extension, the topology on $\Gal(L/K)$ is just the discrete topology. We can now formulate the Galois correspondence for field extensions of infinite degree.
	
	\begin{theorem}[Galois correspondence]
		\label{thm:GaloisCorrespondence}
		Let $L/K$ be a Galois extension. The assignment $H \mapsto L^H$ gives a bijection between
		\begin{enumerate}[label = (\roman*)]
			\item the closed subgroups of $\Gal(L/K)$ and the intermediate field extensions of $L/K$.
			\item the open subgroups of $\Gal(L/K)$ and the finite degree intermediate extensions of $L/K$.
			\item the normal open subgroups of $\Gal(L/K)$ and the finite degree intermediate Galois extensions of $L/K$, in which case we have an isomorphism $\Gal(L/K)/H \to \Gal(L^H/K): \sigma H \mapsto \sigma|_{L^H}$.
		\end{enumerate}
	\end{theorem}

	Now let $G$ be a group and $L/K$ a Galois extension. Whenever we have a group morphism $\phi:\Gal(L/K) \to G$ we will say it is \textit{continuous} if it is so for the Krull topology on $\Gal(L/K)$ and the discrete topology on $G$. Note that this implies that for any subgroup $H \leq G$, the inverse image $\phi^{-1}(H)$ is an open subgroup of $\Gal(L/K)$ and thus by the Galois correspondence fixes a finite degree intermediate extension of $L/K$.\\
	
	Next, we give an overview of the notions of Galois cohomology that will be used throughout this paper. In general, given a Galois extension of fields $L/K$ and an object $X$ defined over $L$, Galois cohomology can be used to classify the objects defined over $K$ which become isomorphic to $X$ when extended to $L$. This method is also known as \textit{Galois descent}. The objects in question can be many structures among which Lie algebras. For the purpose of this paper, we will define everything for Lie algebras, but for a general discussion, we refer to the books \cite{berh10-1} and \cite{serr02-1}. 
	
	We will always assume the fields to be of characteristic zero and the Lie algebras of finite dimension. We do allow field extensions to be of infinite degree. Recall that given a field extension $L/K$ and a Lie algebra $\n^K$ defined over the field $K$, the tensor product of $K$-vector spaces $\n^L := \n^K \otimes_K L$ has a natural Lie algebra structure (over the field $L$) defined by $$[v \otimes l, v' \otimes l'] = [v, v'] \otimes ll'$$ for any $v,v' \in \n^K$ and $l,l' \in L$.
	
	\begin{definition}
		Let $L/K$ be a field extension and $\n^L$ a Lie algebra defined over the field $L$. Then we call a Lie algebra $\m^K$ defined over the field $K$ a \textit{$K$-form} of $\n^L$ if the Lie algebra $\m^L = \m^K \otimes L$ is isomorphic to $\n^L$. Two $K$-forms of $\n^L$ are called equivalent if they are isomorphic over $K$. We write $\mathcal{F}_{K}(\n^L)$ for the set of equivalence classes of $K$-forms of $\n^L$.
	\end{definition} 
	
	Let $L/K$ be a Galois extension and $\n^K$ a Lie algebra defined over the ground field $K$. Note that we have a natural inclusion $\n^K \hookrightarrow \n^L: v \mapsto v \otimes 1$ and an action of the Galois group $\Gal(L/K)$ on $\n^L$, which fixes $\n^K$ seen as a subset of $\n^L$. This action is for any $\sigma \in \Gal(L/K)$, $v \in \n^K$ and $l \in L$ defined by $^\sigma (v \otimes l) := v \otimes \sigma(l)$ and by extending this additively to any element of $\n^L$. As one can check, this action satisfies for all $v,w \in \n^L$, $l \in L$ and $\sigma \in \Gal(K/L)$:
	\begin{enumerate}[label = (\roman*)]
		\item \label{item:prop1} $\prescript{\sigma}{}{(l v)} = \sigma(l)  \prescript{\sigma}{}{v}$,
		\item \label{item:prop2} $\prescript{\sigma}{}{(v + w)} = \prescript{\sigma}{}{v} + \prescript{\sigma}{}{w}$,
		\item \label{item:prop3} $\prescript{\sigma}{}{[v,w]} = [\prescript{\sigma}{}{v}, \prescript{\sigma}{}{w}]$.
	\end{enumerate}
	Conditions \ref{item:prop1} and \ref{item:prop2} tell us that the map $v \mapsto \asi{v}$ is a so-called \textit{semi-linear map} on $\n^L$. We can thus say that $\Gal(L/K)$ has an action on $\n^L$ by semi-linear maps. 
	
	If $\n^K$ and $\mathfrak{m}^K$ are two Lie algebras defined over the field $K$, we denote the set of Lie algebra homomorphisms (over $L$) from $\n^L = \n^K \otimes_K L$ to $\mathfrak{m}^L = \mathfrak{m}^K \otimes_K L$ by $\Hom(\n^L, \m^L)$. The actions of $\Gal(L/K)$ on $\n^L$ and $\mathfrak{m}^L$ as described above now also induce an action of $\Gal(L/K)$ on $\Hom(\n^L, \m^L)$ by setting for any $\varphi \in \Hom(\n^L, \m^L), \sigma  \in \Gal(L/K)$ and $v \in \n^L$:
	\[ (\prescript{\sigma}{}{\varphi}) (v) := \prescript{\sigma}{}{\left(\varphi\left(\prescript{\sigma^{-1}}{}{v}\right)\right)}. \]
	From the properties \ref{item:prop1}, \ref{item:prop2} and \ref{item:prop3} above it follows that $\prescript{\sigma}{}{\varphi}$ is indeed again a Lie algebra homomorphism. Note that if $\mathfrak{p}^K$ is a third Lie algebra over $K$, the equality $\asi{(\phi \varphi)} = \asi{\phi} \asi{\varphi}$ holds for any $\varphi \in \Hom(\n^L, \m^L)$ and $\phi \in \Hom(\m^L, \mathfrak{p}^L)$. In particular, the action of $\Gal(L/K)$ on the invariant subset $\Aut(\n^L) \subset \Hom(\n^L, \n^L)$ is one by group automorphisms, where the group operation on $\Aut(\n^L)$ is composition. In general, when $\Gal(L/K)$ has an action on a group by group automorphisms, we will call this group a \textit{$\Gal(L/K)$-group}.
	
	\begin{definition}
		Let $L/K$ be a Galois extension and $G$ a $\Gal(L/K)$-group. A continuous map $$\rho: \Gal(L/K) \to G,$$ where we write $\rho_\sigma = \rho(\sigma)$ for simplicity, is called a \textit{cocycle} if it satisfies the relation $\rho_{\sigma \tau} = \rho_\sigma \prescript{\sigma}{}{ \rho_\tau}$ for all $\sigma, \tau \in \Gal(L/K)$. The set of cocycles is denoted with $Z^1(L/K, G)$. Two cocycles $\rho, \eta \in Z^1(L/K, G)$ are said to be \textit{equivalent} if there exists a $g \in G$ such that $g \rho_\sigma \asi{g}^{-1} = \eta_\sigma$ for all $\sigma \in \Gal(L/K)$. The set of equivalence classes of cocycles is denoted with $H^1(L/K, G)$ and is called the \textit{first Galois cohomology set}.
	\end{definition}

	\begin{remark}
		The first Galois cohomology set is never empty, as we always have the equivalence class of the \textit{trivial cocycle} $\Gal(L/K) \to G:\sigma \mapsto e_G$. In fact this turns $H^1(L/K, G)$ into a pointed set, with the class of the trivial cocycle being the distinguished element. We will call $H^1(L/K,G)$ trivial if it only consists of this one element.
	\end{remark}

	\noindent Note that, by our previous discussion, if we start with a Lie algebra $\n^K$ defined over $K$, we have a natural action of $\Gal(L/K)$ on $\Aut(\n^L)$ by group automorphisms. Therefore we can talk about the associated first Galois cohomology set $H^1(L/K, \Aut(\n^L))$. Now we discuss the connection between the first Galois cohomology set and the $K$-forms of $\n^L$.
	
	We can associate to each cocycle $\rho:\Gal(L/K) \to \Aut(\n^L)$ a $K$-form $\n^K_\rho \subset \n^L$ by defining
	\begin{equation}
		\label{eq:formByFixedVectors}
		 \n^K_\rho := \{ v \in \n^L \mid \forall \sigma \in \Gal(L/K): \rho_\sigma(\asi{v}) = v \}.
	\end{equation}
	From properties \ref{item:prop1}, \ref{item:prop2}, \ref{item:prop3} and the fact that each $\rho_\sigma$ is an automorphism of $\n^L$, it follows that $\n^K_\rho$ is closed under $K$-scalar multiplication, addition and taking the Lie bracket. Therefore we have that $\n_\rho^K$ is indeed a Lie algebra over $K$. We still need to check that $\n_\rho^K \otimes L \cong \n^L$. For this consider the map $\n_\rho^K \otimes L \to \n^L: v \otimes l \mapsto lv$. It is a standard result that this map is an $L$-vector space isomorphism as proven for example in \cite[Lemma III.8.21.]{berh10-1}. It is then also straightforward to check this map preserves the Lie bracket. As it turns out, the $K$-Lie algebras constructed in this way are all the possible $K$-forms of $\n^L$ up to $K$-isomorphism.
	
	\begin{theorem}[Galois descent for Lie algebras]
		\label{thm:GaloisDescentLieAlgebras}
		Let $L/K$ be a Galois extension and $\n^K$ a Lie algebra defined over $K$. The map
		\[ H^1\left(L/K, \Aut\left(\n^L\right)\right) \to \mathcal{F}_{K}\left(\n^L\right): [\rho] \mapsto \left[\n^K_\rho\right] \]
		is a bijection, which sends the trivial cocycle to $[\n^K]$.
	\end{theorem}

	\begin{proof}
		See \cite[Theorem 1.3 and 1.4]{gs84-1} or \cite[Proposition III.9.1., Remark III.9.2. and Remark III.9.8.]{berh10-1} for a proof of this statement.
	\end{proof}
	
	To see how the inverse of this map works, let $\m^K$ be a $K$-form of $\n^L$. By definition there exists an isomorphism $f:\m^K \otimes_K L \to  \n^L$ of Lie algebras defined over $L$. Then we can associate to $\m^K$ a cocycle $\rho^{\m^K} \in Z^1(L/K, \Aut(\n^L))$ defined by
	\[ \rho_\sigma^{\m^K} = f \left(\asi{f}\right)^{-1}. \]
	for all $\sigma \in \Gal(L/K)$ and $v \in \n^L$. Of course the cocycle $\rho^{\m^K}$ depends on the choice of isomorphism $f$, but its class $[\rho^{\m^K}]$ in $H^1\left(L/K, \Aut\left(\n^L\right)\right)$ does not. In fact this class only depends on the $K$-isomorphism equivalence class of $\m^K$. The inverse of the map from Theorem \ref{thm:GaloisDescentLieAlgebras} is then given by the assignment $\left[\m^K\right] \mapsto \left[\rho^{\m^K}\right]$.
	
	We now present two well-known examples of groups for which the first Galois cohomology set is trivial. First, consider the general linear group $\GL_n(L)$ with the coefficient-wise $\Gal(L/K)$-action: $\asi{(a_{ij})_{ij}} := (\sigma(a_{ij}))_{ij}$ for any $(a_{ij})_{ij} \in \GL_n(L)$ and $\sigma \in \Gal(L/K)$. From the fact that $\sigma$ preserves the addition and multiplication in $L$, it follows that this is an action by group automorphisms on $\GL_n(L)$.
	
	\begin{theorem}[Generalized Hilbert's theorem 90]	\label{thm:GenHilbert90} Let $L/K$ be a Galois extension and consider the general linear group $\GL_n(L)$. Then $H^1(L/K, \GL_n(L))$ is trivial, i.e. it contains only one element.
	\end{theorem}

	\begin{proof}
		This is exactly \cite[p.122, Lemma 1]{serr02-1}.
	\end{proof}

	\noindent Note that $\GL_n(L)$ is also the automorphism group of the abelian Lie algebra of dimension $n$ defined over $L$. Using Theorem \ref{thm:GaloisDescentLieAlgebras}, this agrees exactly with the fact that an abelian Lie algebra over $L$ has only one $K$-form up to $K$-isomorphism.
	
	Secondly, consider the additive group $L_a$ of the field $L$. We can define a $\Gal(L/K)$-action on $L_a$ by $\asi{l} := \sigma(l)$ for any $l \in L$ and $\sigma \in \Gal(L/K)$. It is clear that this is an action by group automorphisms.
	
	\begin{theorem}
		\label{thm:CohAddGroup}
		Let $L/K$ be a Galois extension and $L_a$ the additive group of the field $L$. Then $H^1(L/K, L_a)$ is trivial.
	\end{theorem}

	\begin{proof}
		This is exactly \cite[p.72, Proposition 1]{serr02-1}.
	\end{proof}
	
	So far we only considered Galois extensions $L/K$ which are by definition algebraic. The following result tells us that it is enough to consider algebraic extensions in order to describe all $K$-forms. We will denote by $\overline{K}$ the algebraic closure of $K$. In the remainder of the paper, every algebraic extension of $K$ will be taken as a subfield of $\overline{K}$. This convention in particular implies that isomorphic field extensions of $K$ are in fact identical.
	
	\begin{prop}
		\label{prop:ExistsIsoAlgClosure}
		Let $L/K$ be any (not necessarily algebraic) field extension and $\n^K$, $\m^K$ Lie algebras over $K$. If $\n^K \otimes L \cong \m^K \otimes L$, then so is $\n^K \otimes \overline{K} \cong \m^K \otimes \overline{K}$ and moreover, there exists a finite degree field extension $N/K$ such that $\n^K \otimes N \cong \m^K \otimes N$.
	\end{prop}

	\begin{proof}
		This is a consequence of Hilbert's Nullstellensatz, see \cite[p.124, (i)]{gs84-1} or \cite[Lemma 3.3.]{sulc21-1} for more details on how to prove this.
	\end{proof}
		
	Let us apply this result for $L = \C$, $K = \Q$ and a rational Lie algebra $\n^\Q$. Let $\n^{\Qbar}$, $\n^\R$ and $\n^\C$ denote the Lie algebras $\n^\Q \otimes \Qbar$, $\n^\Q \otimes \R$ and $\n^\Q \otimes \C$, respectively. Proposition \ref{prop:ExistsIsoAlgClosure} now gives us a bijection $\mathcal{F}_\Q(\n^{\Qbar}) \to \mathcal{F}_\Q(\n^\C): [\m^\Q] \mapsto [\m^\Q]$. Together with Theorem \ref{thm:GaloisDescentLieAlgebras}, this gives us a bijection
	\begin{equation}
		\label{eq:GaloisDescentOverC}
		H^1\left(\Qbar/\Q, \Aut\left(\n^{\Qbar}\right)\right) \to \mathcal{F}_\Q\left(\n^\C\right): \left[\rho \right] \mapsto \left[\n_\rho^\Q\right]
	\end{equation}
	which is a useful tool to classify the rational forms of a complex Lie algebra. 
	
	If we want a similar result for the rational forms of a real Lie algebra, the question arises which equivalence classes of $H^1(\Qbar/\Q, \Aut(\n^{\Qbar}))$ are mapped into $\mathcal{F}_\Q(\n^\R)$ under the bijection (\ref{eq:GaloisDescentOverC}) where we use that $\mathcal{F}_\Q(\n^\R)$ is a subset of $\mathcal{F}_\Q(\n^\C)$. Note that if we view $\Qbar$ as a subfield of $\C$, we have a continuous morphism $\nu:\Gal(\C/\R) \to \Gal(\Qbar/\Q): \sigma \mapsto \sigma|_{\Qbar}$. This gives a map
	\begin{equation}
		\omega:H^1\left(\Qbar/\Q, \Aut\left(\n^{\Qbar}\right)\right) \to H^1\left(\C/\R, \Aut\left(\n^\C\right)\right): [\rho] \mapsto [\rho \circ \nu].
	\end{equation}
	Now take any cocycle $\rho \in Z^1(\Qbar/\Q, \Aut(\n^{\Qbar}))$. From \cite[Theorem 1.4]{gs84-1}, it follows that we have the following equivalences
	\begin{align}
		\label{eq:equivalenceFormsOfRealLieAlgebra}
		\nonumber \left[\n^\Q_\rho\right] \in \mathcal{F}_\Q(\n^\R) &\Leftrightarrow \n^\Q_\rho \otimes \R \cong \n^\R\\
		&\Leftrightarrow \omega([\rho]) = [1] .
	\end{align}
	This tells us exactly what elements of $H^1(\Qbar/\Q, \Aut(\n^{\Qbar}))$ are being mapped into $\mathcal{F}_\Q(\n^\R)$ under the bijection (\ref{eq:GaloisDescentOverC}), namely the inverse image under $\omega$ of the class of the trivial cocycle $[1] \in H^1(\C/\R, \Aut(\n^\C))$, allowing us to classify the rational forms of a given real Lie algebra.
	
	\subsection{Galois cohomology of certain semi-direct products}
	\label{sec:galoisCohSemiDirectProd}
	
	In what follows, we prove a result on the Galois cohomology of a certain class of semi-direct products including wreath products. This result will be useful when calculating the Galois cohomology of the automorphism group of a Lie algebra associated to a graph in Section \ref{sec:classKformsL}. Let $L/K$ be a Galois extension. First, note that if $G_1$ and $G_2$ are two $\Gal(L/K)$-groups and $f:G_1 \to G_2$ is a $\Gal(L/K)$-equivariant group morphism, then we get a well-defined induced map on cohomology which we write as $f_\ast$ and is defined by
	\[f_\ast:H^1(L/K, G_1) \to H^1(L/K, G_2): [\rho] \mapsto [f \circ \rho].\]
	Let $L/K$ be a Galois extension and $A$ a finite group with a left action on the set $\{1, \ldots, n\}$ which we write as $a \cdot i$ for any $a \in A$ and $i \in \{1, \ldots, n\}$. Let $G_1, \ldots, G_n$ be $\Gal(L/K)$-groups such that $G_i = G_j$ for any $i \in \{1, \ldots, n\}$ and $j \in A\cdot i$. We define the semi-direct product $\left(\prod_{i = 1}^n G_i\right) \rtimes A$ by letting $A$ act on $\prod_{i = 1}^n G_i$ according to the law
	\[ a \cdot (g_1, \ldots, g_n) := \left( g_{a^{-1} \cdot 1}, \ldots, g_{a^{-1} \cdot n} \right) \]
	for any $a \in A$ and $g_i \in G_i$. Note that if all groups $G_i$ are equal, this group is just a wreath product, which are studied a lot in the literature. We can turn this group into a $\Gal(L/K)$-group as follows. We endow $A$ with the trivial $\Gal(L/K)$ action and $\prod_{i = 1}^n G_i$ with the induced component wise left $\Gal(L/K)$-action, i.e. $\asi{(g_1, \ldots, g_n) = (\asi{g_1}, \ldots , \asi{g_n})}$ for all $g_i \in G_i$ and $\sigma \in \Gal(L/K)$. Clearly, these are actions by group automorphisms. Note that the actions of $A$ and $\Gal(L/K)$ on $\prod_{i = 1}^n G_i$ commute:
	\[ \asi{\left(a \cdot(g_1, \ldots, g_n)\right)} =  \asi{(g_{a^{-1} \cdot 1}, \ldots, g_{a^{-1} \cdot n})} = (\asi{g_{a^{-1} \cdot 1}}, \ldots, \asi{g_{a^{-1} \cdot n}}) = a\cdot(\asi{g_1}, \ldots, \asi{g_n}) = a\cdot\left( \asi{(g_1, \ldots, g_n)}\right). \]
	At last we define a $\Gal(L/K)$-action on $\left(\prod_{i = 1}^n G_i\right) \rtimes A$ by $\asi{(g, a)} = (\asi{g}, \asi{a}) = (\asi{g}, a)$ for all $g \in \prod_{i = 1}^n G_i, a \in A$ and $\sigma \in \Gal(L/K)$. This is an action by automorphisms since
	\begin{align*}
		\asi{\left( (g, a)(h,b) \right)} &= \asi{(g a \cdot h, ab)}\\
		&= \left( \asi{(g a \cdot h)}, ab \right)\\
		&= \left( \asi{g} \asi{(a \cdot h)}, ab \right)\\
		&= \left( \asi{g} a \cdot \asi{h}, ab \right)\\
		&= \left( \asi{g}, a \right)\left( \asi{h}, b \right)\\
		&= \asi{(g,a)}\asi{(h, b)}
	\end{align*}
	for any $g,h \in \prod_{i = 1}^n G_i$ and $a,b \in A$.
	
	We can thus talk about the first Galois cohomology set $H^1(L/K, (\prod_{i = 1}^n G_i) \rtimes A)$. Let $\pi:\left(\prod_{i = 1}^n G_i\right) \rtimes A \to A: (g, a) \mapsto a$ denote the projection morphism and $\iota:A \to \left(\prod_{i = 1}^n G_i\right) \rtimes A: a \mapsto \left((1, \ldots, 1), a\right)$ the natural injection. Note that both $\pi$ and $\iota$ are $\Gal(L/K)$-equivariant. By consequence we have the well-defined maps on cohomology:
	\begin{align*}
	\pi_\ast&: H^1\left(L/K, \left(\prod_{i = 1}^n G_i \right) \rtimes A\right) \to H^{1}(L/K, A): [\alpha] \mapsto [\pi \circ \alpha]\\
	\iota_\ast&: H^1(L/K, A) \to H^1\left(L/K, \left(\prod_{i = 1}^n G_i\right) \rtimes A\right): [\alpha] \mapsto [\iota \circ \alpha].
	\end{align*}
	The set $H^1(L/K, A)$ is relatively easy to understand since $A$ is a group with trivial Galois action, so it is given by all continuous group morphisms from $\Gal(L/K)$ to $A$ up to conjugation by an element of $A$.
	
	\begin{theorem}
		\label{thm:GalCohSemiDirectPerm}
		Let $\left( \prod_{i = 1}^n G_i \right) \rtimes A$ be the $\Gal(L/K)$-group as defined above. Assume that for any finite degree intermediate extension $N/K$ and any $i \in \{1, \ldots, n\}$, the first Galois cohomology set $H^1(L/N, G_i)$ is trivial. Then we have that
		\[ \iota_\ast: H^{1}\left(L/K, A\right) \to H^1\left(L/K, \left(\prod_{i = 1}^n G_i\right) \rtimes A\right): [\rho] \mapsto [\iota \circ \rho] \]
		is a bijection with inverse $\pi_\ast$.
	\end{theorem}
	
	\begin{proof}
		Since $\pi \circ \iota = \text{Id}_A$, it follows that $\pi_\ast \circ \iota_\ast$ is the identity on $H^1(L/K, A)$ as well. For the other direction, we need to prove that $[\rho] = [\iota \circ \pi \circ \rho]$ for an arbitrary $[\rho] \in H^1(L/K, (\prod_{i = 1}^n G_i) \rtimes A)$.
		
		Take any cocycle $\rho:\Gal(L/K) \to (\prod_{i = 1}^n G_i) \rtimes A$. Let us write $\rho_\sigma = (g_\sigma, a_\sigma)$ for all $\sigma \in \Gal(L/K)$. We then have that
		\[ \rho_{\sigma \tau} = \rho_\sigma \asi{\rho}_\tau = (g_\sigma \, a_\sigma \cdot \asi{g}_\tau, a_\sigma a_\tau). \]
		Thus we have $g_{\sigma \tau} = g_\sigma \, a_\sigma \cdot \asi{g}_\tau$ and $a_{\sigma \tau} = a_\sigma a_\tau$. Let us denote for $h \in \prod_{i = 1}^n G_i$ by $(h)_i$ the $i$-th entry of $h$. Then we have
		\[ (g_{\sigma \tau})_i = (g_\sigma \, a_\sigma \cdot \asi{g}_\tau)_i = (g_\sigma)_i  \asi{(g_\tau)}_{a_{\sigma}^{-1}\cdot i} \]
		Note that the group morphism $\pi \circ \rho:\Gal(L/K) \to A$ induces an action of the Galois group on the set $\{1,  \ldots, n\}$. The stabilizers $\text{stab}(i)$ for some $i \in \{1, \ldots, n\}$ are subgroups of $\Gal(L/K)$. Moreover they are open subgroups since they can be written as the inverse image of $\{ a \in A \mid a\cdot i = i \}$ under the continuous map $\pi \circ \rho$ and $A$ is endowed with the discrete topology. By consequence we have $\text{stab}(i) = \Gal(L/L^{\text{stab}(i)})$.
		
		Now note that for $\sigma,\tau \in \text{stab}(i)$, we have
		\[ (g_{\sigma \tau})_{i} = (g_\sigma)_{i} \asi{(g_\tau)}_{i}. \]
		This shows that the assignment $\sigma \mapsto  (g_\sigma)_{i}$ is a cocycle from $\Gal\left(L/L^{\text{stab}(i)}\right)$ to $G_i$. By assumption, $H^1(L/L^{\text{stab}(i)}, G_i)$ is trivial and thus there exists a $h_i \in G_i$ such that $h_i (g_\sigma)_{i} \asi{h}_{i}^{-1} = 1$ for all $\sigma \in \text{stab}(i)$. This gives an element $h = (h_1, \ldots, h_n) \in \prod_{i = 1}^n G_i$. Then define a new cocycle $\tilde{\rho}:\Gal(L/K) \to (\prod_{i = 1}^n G_i) \rtimes A: \sigma \mapsto (h, 1)(g_\sigma, a_\sigma)\asi{(h, 1)}^{-1}$. By the way this $\tilde{\rho}$ is defined, it is clear that $[\rho] = [\tilde{\rho}]$. We also have that
		\[ \tilde{\rho}_\sigma = (\underbrace{h g_\sigma \, a_\sigma \cdot \asi{h}^{-1}}_{:= \tilde{g}_\sigma}, a_\sigma). \]
		Note that now, for $\sigma \in \text{stab}(i)$ we have
		\[ (\tilde{g}_\sigma)_i = h_i (g_\sigma)_i \asi{h}_{a^{-1}_{\sigma}\cdot i}^{-1} = h_i (g_\sigma)_i \asi{h}_{i}^{-1} = 1.\]
		
		Next, let us choose from each orbit of the action defined by $\pi \circ \rho\,\, (= \pi \circ \tilde{\rho})$ on $\{1, \ldots, n\}$, exactly one element $m_i$, giving a subset $\{ m_1, \ldots, m_k \} \subset \{ 1, \ldots, n \}$. For $j \in \text{orb}(m_i)$, we now define the element $r_j := (\tilde{g}_\sigma)_j^{-1}$ where $\sigma \in \Gal(L/K)$ is chosen such that $a_\sigma \cdot m_i = j$. This does not depend on the choice of $\sigma$. Indeed, if $\tau \in \Gal(L/K)$ also satisfies $a_\tau \cdot m_i = j$, then we get
		\[ (\tilde{g}_\tau)_j = (\tilde{g}_{\sigma \sigma^{-1} \tau})_j = (\tilde{g}_\sigma)_j \asi{(\tilde{g}_{\sigma^{-1} \tau})}_{a_\sigma^{-1} \cdot j} = (\tilde{g}_\sigma)_j \asi{(\tilde{g}_{\sigma^{-1}\tau})_{m_i}} = (\tilde{g}_\sigma)_j \]
		where we used that $\sigma^{-1} \tau \in \text{stab}(m_i)$ and thus $(\tilde{g}_{\sigma^{-1} \tau})_{m_i} = 1$. This gives an element $r = (r_1, \ldots, r_n) \in \prod_{i = 1}^n G_i$.
		
	 	We now have that
		\begin{equation*}
		(r, 1)\tilde{\rho}_\sigma  \asi{(r, 1)}^{-1} = (r,1)(\tilde{g}_\sigma, a_\sigma) \left(\asi{r}^{-1}, 1\right)
		= (r \tilde{g}_\sigma \, a_\sigma \cdot \asi{r}^{-1}, a_\sigma).
		\end{equation*}
		At last we show that $r \tilde{g}_\sigma \, a_\sigma \cdot \asi{r}^{-1} = (1, \ldots, 1)$ for all $\sigma \in \Gal(L/K)$. Let $j \in \text{orb}(m_i)$ and let $\tau \in \Gal(L/K)$ such that $\tau(m_i) = j$. Then we have for any $\sigma \in \Gal(L/K)$:
		\begin{align*}
		\left(r \tilde{g}_\sigma \, a_\sigma \cdot \asi{r}^{-1} \right)_j &= r_j  (\tilde{g}_\sigma)_j   \asi{r}_{a_\sigma^{-1} \cdot j}^{-1}\\
		&= (\tilde{g}_{\tau})^{-1}_j  (\tilde{g}_\sigma)_j  \asi{\left(\tilde{g}_{\sigma^{-1} \tau}\right)}_{a_\sigma^{-1}\cdot j} \\
		&= (\tilde{g}_{\tau})^{-1}_j  (\tilde{g}_\sigma)_j \asi{\left( (\tilde{g}_{\sigma^{-1}})_{a_\sigma^{-1}\cdot j} \prescript{\sigma^{-1}}{}{(\tilde{g}_{\tau})}_j   \right)}\\
		&= (\tilde{g}_{\tau})^{-1}_j  (\tilde{g}_\sigma)_j \asi{(\tilde{g}_{\sigma^{-1}})}_{a_\sigma^{-1} \cdot j}(\tilde{g}_{\tau})_j \\
		&= (\tilde{g}_{\tau})^{-1}_j (\tilde{g}_{\sigma \sigma^{-1}})_j  (\tilde{g}_{\tau})_j \\
		&= (\tilde{g}_{\tau})^{-1}_j (\tilde{g}_{\tau})_j \\
		&= 1.
		\end{align*}
		By consequence we have that $(r, 1)\tilde{\rho}_\sigma \asi{(r, 1)}^{-1} = ((1, \ldots, 1), a_\sigma)$ for any $\sigma \in \Gal(L/K)$. Note that $(\iota \circ \pi \circ \rho)_\sigma = ((1, \ldots, 1), a_\sigma)$ and thus that we have shown that $[\iota \circ \pi \circ \rho] = [\tilde{\rho}] = [\rho]$.
	\end{proof}

	\section{Galois cohomology of $\Aut(\n^{\overline{K}}_{\G,c})$}
	\label{sec:classKformsL}
	
	Let $\G = (S,E)$ be a graph and $L/K$ a Galois extension of subfields of $\C$. The natural inclusion $\mathfrak{f}^K(S) \hookrightarrow \mathfrak{f}^L(S)$ induces an inclusion $\n_{\G, c}^K \hookrightarrow \n_{\G, c}^L$. Using this inclusion we get an isomorphism $\n_{\G, c}^K \otimes L \cong \n_{\G, c}^L: v \otimes l \mapsto lv$ of Lie algebras over $L$. Thus $\n_{\G, c}^K$ is a distinguished $K$-from of $\n_{\G, c}^L$. As defined in section \ref{sec:GaloisCohomologyDefinitions}, this gives an action of $\Gal(L/K)$ on $\n_{\G, c}^L$ by semi-linear maps which fix the vectors in the $K$-form $\n_{\G, c}^K$. In particular the action fixes the vertices $S \subset \n_{\G, c}^L$. This action induces a $\Gal(L/K)$-action on $\Aut(\n_{\G, c}^L)$ and thus we can talk about the first Galois cohomology set $H^1(L/K, \Aut(\n_{\G, c}^L))$.
	
	The goal of this section is to compute the first Galois cohomology set $H^1(\overline{K}/K, \Aut(\n_{\G,c}^{\overline{K}}))$. In order to do so we will use a known result which reduces this calculation from the full automorphism group to the normalizer of a Cartan subgroup of $\Aut(\n_{\G,c}^{\overline{K}})$, i.e. the Zariski connected centralizer of a maximal torus.
	
	\begin{definition}
		For a map $\Psi:S \to K$, we say an automorphism $f$ of $\n_{\G,c}^K$ is a \textit{vertex-diagonal automorphism determined by $\Psi$} if it holds that $\forall \alpha \in S: f(\alpha) = \Psi(\alpha) \alpha$. We will write $D_{\G,c}$ for the subgroup of $\Aut(\n_{\G, c}^K)$ of all vertex diagonal automorphisms.
	\end{definition}
	
	A nice fact about the Lie algebras $\n_{\G, c}^K$ is that for any choice of a map $\Psi:S \to L$, there is a unique vertex-diagonal automorphism determined by $\Psi$. This follows immediately from the structure of $G$ (Theorem \ref{thm:structureThmG}). We will denote this automorphism with $f_\Psi$. Using this notation we can also write $D_{\G,c} = \{ f_\Psi \mid \Psi:S \to K \}$. It is also a well know fact that $D_{\G,c}$ is a maximal (split) torus of the linear algebraic group $\Aut(\n_{\G, c}^K)$. Let us also write $D_S = p(D_{\G,c})$ for the projection of the vertex diagonal automorphisms onto $V$ which gives all diagonal linear maps with respect to the basis of vertices $S$.
	
	\begin{lemma}
		\label{lem:normalizerOfVertexDiag}
		The normalizer of $D_{\G,c}$ in $\Aut(\n_{\G, c}^K)$ is contained in $T_{\G,c}$.
	\end{lemma}

	\begin{proof}
		Let $g$ be an element of the normalizer of $D_{\G,c}$ in $\Aut(\n_{\G, c}^K)$. Let $f \in D_{\G,c}$ be the vertex diagonal automorphism determined by the constant map $\Psi:S \to L:\alpha \mapsto 2$. By definition of the normalizer we know that $gfg^{-1}$ is again vertex diagonal with the same eigenvalues as $f$ (counted with multiplicities). Thus we know that the eigenspace of $gfg^{-1}$ with eigenvalue 2 has dimension equal to $|S|$. Note that the eigenvalues of $ghg^{-1}$ on the derived algebra $[\n_{\G,c}^K, \n_{\G,c}^K]$ will have to be strictly greater than $2$ since they are equal to the product of two eigenvalues of $gfg^{-1}$. This shows that $V$ is exactly the eigenspace of $gfg^{-1}$ with eigenvalue 2. Note that for any $\alpha \in S$ we have $gfg^{-1}g\alpha = gf\alpha = 2 g\alpha$ and thus that $g$ maps $V$ into the eigenspace of $gfg^{-1}$ with eigenvalue $2$ which we just showed is equal to $V$. This proves that $g(V) = V$ and thus that $g \in T_{\G,c}$.
	\end{proof}
	
	In particular, the above lemma implies that the centralizer of $D_{\G, c}$ in $\Aut(\n_{\G, c}^K)$ is contained in $T_{\G,c}$. From the structure of $G = p(T_{\G,c})$ it follows immediately that $D_{\G,c}$ is its own centralizer. We thus have that $D_{\G, c}$ is a Cartan subgroup of $\Aut(\n_{\G,c}^K)$. The following result proven in \cite[Chapter 3, Lemma 6]{serr02-1} will thus come in handy.
	
	\begin{lemma}
		\label{lem:galoisCohCartanSubgroup}
		Let $C$ be a Cartan subgroup of a linear algebraic group $A$ defined over a perfect field $K$, and let $N$ be the normalizer of $C$ in $A$. The canonical map $H^1(\overline{K}/K, N) \to H^1(\overline{K}/K, A)$ induced by the inclusion is surjective.
	\end{lemma}

	For any field $K \subset \C$, let us now fix an embedding $i:\Aut(\overline{\G}) \to \Aut(\n_{\G,c}^K)$ by
	\begin{equation}
		\label{eq:morphismI}
		i := \tilde{p}^{-1} \circ \overline{P} = \tilde{p}^{-1} \circ P \circ r.
	\end{equation}
	Recall that $\tilde{p}$ is the morphism that maps any graded automorphism of $\n_{\G,c}$ to its restriction to $V$ (see (\ref{eq:definitionptilde})), that $P$ is the morphism that maps a graph automorphism to the corresponding permutation matrix on $V$ w.r.t. the basis $S$ and that $r$ is the splitting morphism $r:\Aut(\overline{\G}) \to \Aut(\G)$ that was chosen in section \ref{sec:quotientGraph}. We can also define a map $\pi: \Aut(\n_{\G, c}^K) \to \Aut(\overline{\G})$ by
	\begin{equation}
		\label{eq:morphismPi}
		\pi := q \circ p
	\end{equation}
	where $q$ and $p$ are the morphisms as defined by (\ref{eq:definitionq}) and (\ref{eq:morphismp}). Note that $\pi$ defines a left inverse to $i$, indeed $\pi \circ i = q \circ p \circ \tilde{p}^{-1} \circ \overline{P} = q \circ \overline{P} = \Id$. 
	
	Now we are ready to calculate the Galois cohomology $H^1(\overline{K}/K, \Aut(\n_{\G, c}^{\overline{K}}))$. The subgroups $T_{\G,c}$ and $D_{\G,c}$ are first of all algebraic (over $\Q$) and thus can be defined over any subfield of $\C$ and second they are $\Gal(L/K)$-invariant when seen as subgroups of $\Aut(\n_{\G,c}^L)$. When working over the field $L$, the group $\GL(V)$ can be identified with the group $\GL_{|S|}(L)$ by using the basis of vertices $S$. In this way $\GL(V)$ becomes a $\Gal(L/K)$ group. The subgroup $G$ will be $\Gal(L/K)$-invariant and thus a $\Gal(L/K)$-group itself. This action of $\Gal(L/K)$ on $G$ is the same one you would get from the action on $T_{\G,c}$ via the group isomorphism $\tilde{p}:T_{\G,c} \to G$.
	
	\begin{theorem}
		\label{thm:GaloisCohAutnG}
		Let $K \subset \C$ be a field and endow $\Aut(\overline{\G})$ with the trivial $\Gal(\overline{K}/K)$-action. Then the map $i_\ast:H^1(\overline{K}/K, \Aut(\overline{\G})) \to H^1(\overline{K}/K, \Aut(\n_{\G, c}^{\overline{K}}))$ is a bijection with inverse $\pi_\ast$.
	\end{theorem}
	
	\begin{proof}
		First note that since $\pi$ is a left inverse to the map $i$, we also get for the induced maps on the cohomology that $(\pi \circ p)_\ast \circ i_\ast = \Id$. This already proves that $i_\ast$ is injective and that if $i_\ast$ is a bijection, its inverse must be given by $\pi_\ast$. Let us now argue for the surjectivity of $i_\ast$.
		
		As discussed above, the subgroup of vertex diagonal automorphisms $D_{\G,c} \leq \Aut(\n_{\G, c}^{\overline{K}})$ is a maximal torus and in fact also a Cartan subgroup of $\Aut(\n_{\G,c}^{\overline{K}})$. By Lemma \ref{lem:normalizerOfVertexDiag}, the normalizer of $D_{\G,c}$ in $\Aut(\n_{\G,c}^{\overline{K}})$ is contained in $T_{\G,c}$ and since $\tilde{p}:T_{\G,c} \to p(T_{\G,c}) = G$ is an isomorphism of algebraic groups, we can look at the normalizer of $D_S = p(D_{\G,c})$ in $G$. This normalizer is given by $D_S \cdot \Aut(\G) \leq G$. Using Lemma \ref{lem:galoisCohCartanSubgroup} and the fact that subfields of $\C$ are always perfect, we thus find that the map
		\[ (\tilde{p}^{-1})_\ast: H^1\left(\overline{K}/K, D_S \cdot P\left(\Aut(\G)\right)\right) \to H^1\left(\overline{K}/K, \Aut\left(\n_{\G,c}^{\overline{K}}\right)\right) \] is
		surjective. But then it follows immediately from the inclusion $D_S \cdot P(\Aut(\G) \subset \left(\prod_{\lambda \in \Lambda} \GL(V_\lambda) \right) \cdot \overline{P}(\Aut(\overline{\G})$ that also the map
		\begin{equation}
			\label{eq:pInverseIndMapSurj}
			(\tilde{p}^{-1})_\ast: H^1\left(\overline{K}/K, \left(\prod_{\lambda \in \Lambda} \GL(V_\lambda) \right) \cdot \overline{P}(\Aut(\overline{\G})) \right) \to H^1\left(\overline{K}/K, \Aut\left(\n_{\G,c}^{\overline{K}}\right)\right)
		\end{equation}
		is surjective.
		
		Recall that in order to define the morphism $r$, we fixed an ordering of the vertices in each coherent component $\lambda = \{v_{\lambda, 1}, \ldots, v_{\lambda, |\lambda|}\}$. Let us also order the coherent components themselves $\Lambda = \{\lambda_1, \ldots, \lambda_n\}$ with $n = |\Lambda|$ and write $n_j = |\lambda_j|$. We have isomorphisms of groups $\GL_{n_j}( \overline{K}) \to \GL(V_{\lambda_j}): B \mapsto \overline{B}$, where $\overline{B}$ is the linear map which has matrix representation $B$ with respect to the basis $v_{{\lambda_j}, 1}, \ldots, v_{{\lambda_j}, n_j}$. Endow $\GL_{n_j}(\overline{K})$ with the coefficient-wise $\Gal(\overline{K}/K)$-action. Note that since we ordered the coherent components, the group $\Aut(\overline{\G})$ has an action on the set $\{1, \ldots, n\}$. Thus we can define the $\Gal(\overline{K}/K)$-group $\displaystyle \left( \prod_{j = 1}^n \GL_{n_j}(\overline{K}) \right) \rtimes \Aut(\overline{\G})$ in the same way as we did in section \ref{sec:galoisCohSemiDirectProd} for $A = \Aut(\overline{\G})$ and $G_j = \GL_{n_j}(\overline{K})$. As $\Gal(\overline{K}/K)$-groups, $\displaystyle \left( \prod_{j = 1}^n \GL_{n_j}(\overline{K}) \right) \rtimes \Aut(\overline{\G})$ and $\displaystyle \left(\prod_{\lambda \in \Lambda} \GL(V_\lambda) \right) \cdot \overline{P}(\Aut(\overline{\G}))$ are isomorphic by sending $((B_1, \ldots, B_{n}), \varphi)$ to $\overline{B}_1 \cdot \ldots \cdot \overline{B}_{n} \cdot \overline{P}(\varphi)$. By Theorem \ref{thm:GenHilbert90}, $H^1(L/K, \GL_{n_j}(L))$ is trivial for any Galois extension $L/K$ and thus we can apply Theorem \ref{thm:GalCohSemiDirectPerm} to get a bijection
		\[ \overline{P}_\ast:H^1(\overline{K}/K, \Aut(\overline{\G})) \to H^1\left(\overline{K}/K, \left(\prod_{\lambda \in \Lambda} \GL(V_\lambda) \right) \cdot \overline{P}(\Aut(\overline{\G})) \right). \]
		Composing this map with the surjective map from (\ref{eq:pInverseIndMapSurj}), we find that \[i_\ast:H^1(L/K, \Aut(\overline{\G})) \to H^1(\overline{K}/K, \Aut(\n_{\G, c}^{\overline{K}}))\]
		is surjective, which concludes the proof.
	\end{proof}
	
	Note that the set \(H^1 \left(\overline{K}/K, \Aut(\overline{\G})\right) \) is relatively easy to understand. It is equal to the set of equivalence classes of actions of $\Gal(\overline{K}/K)$ on $\overline{\G}$ by automorphisms of the quotient graph, where two actions are equivalent if they are conjugated by a fixed automorphism of $\overline{\G}$. Hence Theorem \ref{thm:GaloisCohAutnG} gives us a complete understanding of $H^1(\overline{K}/K, \Aut(\n_{\G, c}^{\overline{K}}))$ from only information of the graph $\G$.
	
	\section{Applications for real and rational forms}
	\subsection{Rational forms of $\n^\C_{\G,c}$ and $\n^\R_{\G,c}$}
	\label{sec:rationalForms}
	
	In this section we give a description of the rational forms of $\n_{\G,c}^\R$ and $\n^\C_{\G,c}$. Consider a Galois extension $L/K$ of subfields of $\C$ and a continuous morphism $\rho:\Gal(L/K) \to \Aut(\overline{\G})$. As discussed in previous section $i \circ \rho$ is a cocycle in $Z^1(L/K, \Aut(\n_{\G, c}^L))$ and thus by (\ref{eq:formByFixedVectors}) gives us a $K$-form of $\n_{\G, c}^L$. For notational purposes, we will simply denote this form as
	\[ \n_{\rho, c}^K = \{ v \in \n^{L}_{\G,c} \mid \forall \sigma \in \Gal(L/K): \, i(\rho_\sigma)(\asi{v}) = v  \} \]
	The following theorem now tells us that all rational forms of $\n_{\G, c}^\C$ arise in this way, and which ones will be rational forms of the real Lie algebra $\n_{\G,c}^\R$. Let us write $\tau \in \Gal(\Qbar/\Q)$ for the complex conjugation automorphism on $\Qbar$.
	
	\begin{theorem}
		\label{thm:rationalFromsGraphRandC}
		Let $\G$ be a simple undirected graph and $c > 1$. We have bijections
		\begin{equation}
			H^1\left(\Qbar/\Q, \Aut(\overline{\G})\right) \to \mathcal{F}_\Q\left(\n_{\G,c}^\C\right): [\rho] \mapsto \left[\n_{\rho,c}^\Q\right]
		\end{equation}
		and
		\begin{equation}
			\left\{ [\rho] \in H^1\left(\Qbar/\Q, \Aut(\overline{\G})\right) \,\, \Big| \tau \in \ker(\rho) \right\} \to \mathcal{F}_\Q\left(\n^\R_{\G,c}\right): [\rho] \mapsto \left[\n_{\rho,c}^\Q\right].
		\end{equation}
	\end{theorem}

	\begin{proof} 
		The first bijection follows immediately from combining Theorem \ref{thm:GaloisCohAutnG} for $K = \Q$, Theorem \ref{thm:GaloisDescentLieAlgebras} and the bijection derived in (\ref{eq:GaloisDescentOverC}). To prove the second bijection, we will apply the equivalence given in (\ref{eq:equivalenceFormsOfRealLieAlgebra}).
		
		If we view $\Qbar$ as a subfield of $\C$, we get a continuous morphism $\nu:\Gal(\C/\R) \to \Gal(\Qbar/\Q):\sigma \mapsto \sigma|_{\Qbar}$. This allows us to define the maps $\omega_1$ and $\omega_2$ with domain and codomain as in the diagram below and which send $[\rho]$ to $[\rho \circ \nu]$. Note that we have also two maps $i_\ast$ in the diagram, induced by the map $i:\Aut(\overline{\G}) \to \Aut(\n_{\G,c}^K)$ as defined in (\ref{eq:morphismI}) for a general field $K \subset \C$.
		\begin{equation}
			\begin{tikzcd}
				H^1(\Qbar/\Q, \Aut(\overline{\G})) \arrow[r, "\omega_1"] \arrow[d, "i_\ast"] &H^1(\C/\R, \Aut(\overline{\G})) \arrow[d, "i_\ast"]\\
				H^1(\Qbar/\Q, \Aut(\n_{\G, c}^{\Qbar})) \arrow[r, "\omega_2"] &H^1(\C/\R, \Aut(\n_{\G,c}^\C)) 
			\end{tikzcd}
		\end{equation}
		The diagram is commutative as $i \circ  (\rho \circ \nu) = (i \circ \rho) \circ \nu$ on the level of representatives. From the equivalence in (\ref{eq:equivalenceFormsOfRealLieAlgebra}) we know that the classes in $H^1(\Qbar/\Q, \Aut(\n^{\Qbar}_{\G,c}))$ that give a rational form of $\n^\R_{\G,c}$ are exactly the classes in the inverse image of the class of the trivial cocycle under $\omega_2$. Now since the diagram above commutes and since by Theorem \ref{thm:GaloisCohAutnG} the induced maps $i_\ast$ are bijections, we get that those classes $[\rho] \in \omega_1^{-1}([1])$ are exactly the ones for which $[\n_{\rho, c}^\Q]$ lies in $\mathcal{F}_\Q(\n^\R_{\G,c})$. 
	
		Since we have that
		\begin{align*}
			[\rho] \in \omega_1^{-1}([1]) &\Leftrightarrow [\rho \circ \nu] = [1]\\
			&\Leftrightarrow \exists \varphi \in \Aut(\overline{\G}): \forall \sigma \in \Gal(\C/\R): \rho_{\nu(\sigma)} = \varphi \Id_{\Lambda} \varphi^{-1}\\
			&\Leftrightarrow \rho_\tau = 1\\
			&\Leftrightarrow \tau \in \ker(\rho),
		\end{align*}
		this proves that the second map is a bijection.
	\end{proof}

	We are now ready to prove Theorem \ref{thm:injectiveVersionClassificationRationalForms} which can be seen as an `injective version' of the above theorem.
	
	\begin{proof}[Proof of Theorem \ref{thm:injectiveVersionClassificationRationalForms}]
		Take any rational form of $\n^\C_{\G,c}$. Then it follows from Theorem \ref{thm:rationalFromsGraphRandC} that up to $\Q$-isomorphism, the form is given by $\n_{\rho, c}^\Q$ for some continuous morphism $\rho:\Gal(\Qbar/\Q) \to \Aut(\overline{\G})$. Note that $\ker(\rho)$ is an open normal subgroup of $\Gal(\Qbar/\Q)$. By Theorem \ref{thm:GaloisCorrespondence} we get that $L_\rho := \Qbar^{\ker(\rho)}$ is a finite degree Galois extension of $\Q$ together with a natural isomorphism of groups $$\Gal(\Qbar/\Q)/\ker(\rho) \to \Gal(L_\rho/\Q): \sigma\ker(\rho)\mapsto \sigma|_\rho.$$ We therefore get an induced injective morphism of groups
 		\[ \overline{\rho}: \Gal(L_\rho/\Q) \cong \Gal(\Qbar/\Q)/\ker(\rho) \to \Aut(\overline{\G}) \]
		which gives a class $[\overline{\rho}] \in H^1(L_\rho/\Q, \Aut(\overline{\G}))$. Note that we have a natural injection $\n_{\G, c}^{L_\rho} \hookrightarrow \n_{\G,c}^{\Qbar}$ and that this injection restricts to a $\Q$-Lie algebra isomorphism $\n_{\overline{\rho}, c}^{\Q} \cong \n_{\rho, c}^\Q$. 
		Following Theorem \ref{thm:rationalFromsGraphRandC}, we also have that $[\n_{\overline{\rho}, c}^\Q] \in \mathcal{F}_\Q(\n^\R_{\G, c})$ if and only if $\tau \in \ker(\rho)$ and thus if and only if $L_\rho$ is a real extension of $\Q$. 
		
		For the second statement, if $\eta:\Gal(\Qbar/\Q)\to \Aut(\overline{\G})$ is another continuous morphism, we have the equivalences
		\begin{align*}
			\n_{\overline{\rho}, c}^\Q \cong \n_{\overline{\eta}, c}^\Q &\Leftrightarrow \n_{\rho, c}^\Q \cong \n_{\eta, c}^\Q\\
			&\Leftrightarrow [\rho] = [\eta]\\
			&\Leftrightarrow \ker(\rho) = \ker(\eta) \text{ and } [\overline{\rho}] = [\overline{\eta}]\\
			&\Leftrightarrow L_\rho = L_\eta \text{ and } [\overline{\rho}] = [\overline{\eta}].
		\end{align*} 
		At last, note that if $L'/\Q$ is any finite degree Galois extension and $\rho':\Gal(L'/\Q) \to \Aut(\overline{\G})$ is any injective group morphism, we can define the continuous morphism $\rho:\Gal(\Qbar/\Q) \to \Aut(\overline{\G}): \sigma \mapsto \rho'_{\sigma|_{L'}}$ which now satisfies $\overline{\rho} = \rho'$. This shows all that needed to be proven.
	\end{proof}

	The following example shows that Theorem \ref{thm:injectiveVersionClassificationRationalForms} can be used to simplify certain classifications of Lie algebras, especially for quadratic extensions.
	
	\begin{example}[Direct sum of two 3-dimensional Heisenberg Lie algebras]
		\label{ex:twoCopiesHeisenberg}
		Consider the graph $\G = (S, E)$ defined by $S = \{ \alpha_1, \alpha_2, \beta_1, \beta_2 \}$ and $E = \{ \{\alpha_1, \beta_1 \}, \{\alpha_2, \beta_2\}\}$. The set of coherent components is then given by $\Lambda = \{ \lambda_1 := \{ \alpha_1, \beta_1 \}, \lambda_2 := \{\alpha_2, \beta_2\} \}$. A figure of the graph and quotient graph are given below.
		
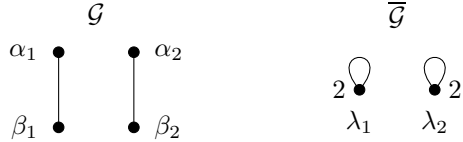
\begin{figure}[H]
			\centering
			\begin{tikzpicture}[every loop/.style={}]
			\draw (0,0) -- (0,1);
			\draw (1,0) -- (1,1);
			\filldraw [black] (0,0) circle (2pt) node[left = 0.15] {$\beta_1$};
			\filldraw [black] (1,0) circle (2pt) node[right = 0.15] {$\beta_2$};
			\filldraw [black] (0,1) circle (2pt) node[left = 0.15] {$\alpha_1$};
			\filldraw [black] (1,1) circle (2pt) node[right = 0.15] {$\alpha_2$};
			\node at (0.5,1.5) {$\G$};

			\filldraw [black] (4,0.5) circle (2pt) node[anchor = east] {2\,} node[below = 0.15] {$\lambda_1$};
			\filldraw [black] (5,0.5) circle (2pt) node[anchor = west] {\,2} node[below = 0.15] {$\lambda_2$};
			\draw[cm={1.5 ,0 ,0 ,1.5 ,(4,0.5)}] (0,0)  to[in=50,out=130, loop] (0,0);
			\draw[cm={1.5 ,0 ,0 ,1.5 ,(5,0.5)}] (0,0)  to[in=50,out=130, loop] (0,0);
			\node at (4.5,1.5) {$\overline{\G}$};
			\end{tikzpicture}
		\end{figure}
		\noindent There are only two automorphisms of the quotient graph, namely the identity and $\varphi \in \Aut(\overline{\G})$ which is defined by $\varphi(\lambda_1) = \lambda_2$ and $\varphi(\lambda_2) = \lambda_1$. We can define the morphism $r:\Aut(\overline{\G}) \to \Aut(\G)$ by letting $r(\varphi)(\alpha_1) = \alpha_2$, $r(\varphi)(\alpha_2) = \alpha_1$, $r(\varphi)(\beta_1) = \beta_2$ and $r(\varphi)(\beta_2) = \beta_1$. The associated 2-step nilpotent Lie algebra $\n^L_\G$ is then isomorphic to a direct sum of two 3-dimensional Heisenberg Lie algebras with basis $\{\alpha_1, \alpha_2, \beta_1, \beta_2,\, \gamma_1, \gamma_2 \}$ where $\gamma_1 := [\alpha_1, \beta_1]$ and $\gamma_2 := [\alpha_2, \beta_2]$.
		
		It is clear that if $\Gal(L/\Q) \to \Aut(\overline{\G})$ is an injective group morphism, $L/\Q$ must have degree 2 or 1. All non-isomorphic degree 2 or 1 Galois extensions of $\Q$ are given by $\Q(\sqrt{d})$ for $d$ a square free non-zero integer. Note that if $d = 1$, $\Q(\sqrt{d}) = \Q$. For all square free non-zero integers $d$, let $\rho_d$ denote the uniquely determined injective group morphism $\rho_d:\Gal(\Q(\sqrt{d})/\Q) \to \Aut(\overline{\G})$. For simplicity, let us write $\n_{d}^\Q$ for the associated rational form $\n_{\rho_d, 2}^\Q$ of $\n_{\G, 2}^\C$. From Theorem \ref{thm:injectiveVersionClassificationRationalForms} we then get that the sets
		\[\{  \n_{d}^\Q \mid d \neq 0 \text{ square free} \} \quad \quad \text{and} \quad \quad \{  \n_{d}^\Q \mid d \geq 1 \text{ square free} \}\]
		give us a complete set of pairwise non-isomorphic rational forms of $\n^\C_{\G, 2}$ and $\n^\R_{\G, 2}$, respectively. Note that $\n_{1}^\Q \cong \n_{\G, 2}^\Q$ is the standard rational form of $\n^\C_{\G, 2}$. We thus get an alternative proof of \cite[Proposition 3.2.]{laur08-1} without using the Pfaffian form on $2$-step nilpotent Lie algebras.
		
		For non-zero square-free $d$, a basis for $\n^\Q_{d} \subset \n^\C_{\G, 2}$ can be given by 
		\begin{alignat*}{3}
		X_1 &:= \alpha_1 + \alpha_2 & \quad \quad Y_1 &:= \beta_1 + \beta_2 & \quad \quad Z_1 &:= \gamma_1 + \gamma_2\\
		X_2 &:= \sqrt{d}(\alpha_1 - \alpha_2) & Y_2 &:= \sqrt{d}(\beta_1 - \beta_2) & Z_2 &:= \sqrt{d}(\gamma_1 - \gamma_2).
		\end{alignat*}
		The bracket relations of the rational Lie algebra $\n^\Q_{d}$ in this basis are then given by
		\begin{alignat*}{2}
		[X_1, Y_1] &= Z_1 &\quad \quad  [X_2, Y_1] &= Z_2\\
		[X_1, Y_2] &= Z_2 &\quad \quad  [X_2, Y_2] &= d\, Z_1. 
		\end{alignat*}	
	\end{example}

	We can also consider the complement graph.

	\begin{example}
		\label{ex:twoCopiesHeisenbergComplement}
		Let $\G = (S,E)$ be the graph from Example \ref{ex:twoCopiesHeisenberg} and $\G^\ast$ its complement graph, i.e. $\G^\ast = (S, E^\ast)$ with $E^\ast = \{ \{\alpha, \beta\} \mid \alpha, \beta \in S, \alpha \neq \beta, \{\alpha, \beta \} \notin E \}$. A figure of $\G^\ast$ and its quotient graph are given below.
		
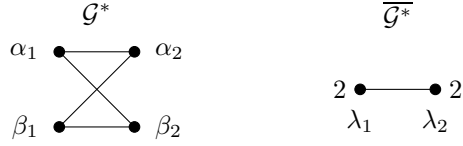
\begin{figure}[H]
			\centering
			\begin{tikzpicture}[every loop/.style={}]
				\draw (0,0) -- (1,1);
				\draw (0,0) -- (1,0);
				\draw (0,1) -- (1,0);
				\draw (0,1) -- (1,1);
				\filldraw [black] (0,0) circle (2pt) node[left = 0.15] {$\beta_1$};
				\filldraw [black] (1,0) circle (2pt) node[right = 0.15] {$\beta_2$};
				\filldraw [black] (0,1) circle (2pt) node[left = 0.15] {$\alpha_1$};
				\filldraw [black] (1,1) circle (2pt) node[right = 0.15] {$\alpha_2$};
				\node at (0.5,1.5) {$\G^\ast$};

				\draw (4,0.5) -- (5,0.5);
				\filldraw [black] (4,0.5) circle (2pt) node[anchor = east] {2\,} node[below = 0.15] {$\lambda_1$};
				\filldraw [black] (5,0.5) circle (2pt) node[anchor = west] {\,2} node[below = 0.15] {$\lambda_2$};
				\node at (4.5,1.5) {$\overline{\G^\ast}$};
			\end{tikzpicture}
		\end{figure}
		\noindent Since $\Aut(\overline{\G^\ast}) = \Aut(\overline{\G})$, it follows that the only injective group morphisms $\Gal(L/\Q) \to \Aut(\G_{\rd}^\ast)$ are the morphisms $\rho_d: \Gal(\Q(\sqrt{d})/\Q) \to \Aut(\G_{\rd})$ from Example \ref{ex:twoCopiesHeisenberg} where $d$ is any square free non-zero integer. Again, let us simply write $\n_{d, \ast}^\Q$ for the associated rational form $\n^\Q_{\rho_d, 2}$ of $\n^\C_{\G^\ast, 2}$. From Theorem \ref{thm:injectiveVersionClassificationRationalForms} it thus follows that the sets
		\[\{  \n_{d, \ast}^\Q \mid d \neq 0 \text{ square free} \} \quad \quad \text{and} \quad \quad \{  \n_{d, \ast}^\Q \mid d \geq 1 \text{ square free} \}\]
		give us a complete set of pairwise non-isomorphic rational forms of $\n^\C_{\G^\ast, 2}$ and $\n^\R_{\G^\ast, 2}$, respectively. A basis for $\n^\C_{\G^\ast, 2}$ can be given by $\{ \alpha_1, \alpha_2, \beta_1, \beta_2, \gamma_1, \gamma_2, \gamma_3, \gamma_4 \}$ where $\gamma_1 := [\alpha_1, \beta_2]$, $\gamma_2 := [\alpha_2, \beta_1]$, $\gamma_3 = [\alpha_1, \alpha_2]$, $\gamma_4 = [\beta_1, \beta_2]$. For non-zero square-free $d$, a basis for the form $\n^\Q_{d, \ast}$ can then be given by:
		\begin{alignat*}{4}
			X_1 &:= \alpha_1 + \alpha_2 & \quad \quad Y_1 &:= \beta_1 + \beta_2 & \quad \quad Z_1 &:= \gamma_1 + \gamma_2 & \quad \quad Z_3 &= -2\sqrt{d} \gamma_3\\
			X_2 &:= \sqrt{d}(\alpha_1 - \alpha_2) & Y_2 &:= \sqrt{d}(\beta_1 - \beta_2) & Z_2 &:= \sqrt{d}(\gamma_2 - \gamma_1) & \quad \quad Z_4 &= -2\sqrt{d} \gamma_4.
		\end{alignat*}
		The bracket relations of the rational Lie algebra $\n^\Q_{d, \ast}$ in this basis are then given by
		\begin{alignat*}{2}
			[X_1, X_2] &= Z_3 \quad \quad & [X_2, Y_1] &= Z_2\\
			[X_1, Y_1] &= -Z_1 \quad \quad & [X_2, Y_2] &= dZ_1\\
			[X_1, Y_2] &= -Z_2 \quad \quad & [Y_1, Y_2] &= Z_4.
		\end{alignat*}
		We thus get an alternative proof of \cite[Proposition 4.5.]{laur08-1}.
	\end{example}

	\subsection{Number of different $\Q$-forms}
	
	We apply Theorem \ref{thm:injectiveVersionClassificationRationalForms} to show that the Lie algebras $\n_{\G,c}^\R$ and  $\n^\C_{\G,c}$ have either exactly one or infinitely many rational forms. In order to prove this, we need a lemma that ensures the existence of enough non-isomorphic cyclic Galois extensions of a certain degree.
	
	\begin{lemma}
		\label{lem:infinitelyManyPrimeGaloisExt}
		For every positive integer $d > 1$, there exist infinitely many real Galois extensions $L_i/\Q$ with $i \in \N$ such that $\Gal(L_i/\Q) \cong \Z/d\Z$ and $L_i \cap L_j = \Q$ for any $i,j \in \N$ with $i \neq j$.
	\end{lemma}

	\begin{proof}
		By Dirichlet's theorem \cite{diri13-1}, there are infinitely many different primes $p_i$ for $i \in \N$ such that $p_i = 1 \mod 2d$ for all $i \in \N$. If we denote by $\zeta_k = e^{2\pi i/k} $ the primitive $k$-th root of unity, we can define the cyclotomic field extensions $K_i = \Q(\zeta_{p_i})$, for which the Galois group $\Gal(K_i/\Q)$ is cyclic of order $p_i - 1$. Since $2d \mid  p_i - 1$, there exists a (unique) cyclic subgroup $H_i \subset \Gal(K_i/\Q)$ of index $2d$. Let $K_i^{H_i}$ denote the field which is fixed under $H_i$. Since $\Gal(K_i/ \Q)$ is abelian, $H_i$ is a normal subgroup and thus we have $\Gal(K_i^{H_i}/\Q) \cong \frac{\Gal(K_i/\Q)}{H_i} \cong \Z/2d\Z$. 
		
		Note that $\Gal(K_i^{H_i}/\Q)$ has a unique element $\sigma$ of order 2 and in case $K_i^{H_i}$ is not totally real, this must be the complex conjugation automorphism. Let $L_i$ be the subfield of $K_i^{H_i}$ which is fixed by $\{ 1, \sigma \}$. As before we have that $\{ 1, \sigma \}$ is a normal subgroup of $\Gal(K_i^{H_i}/\Q)$ and thus that $\Gal(L_i/\Q) \cong \frac{\Gal(K_i^{H_i}/\Q)}{\{1, \sigma\}} \cong \Z/d\Z$. Note that, even if $K_i$ was not totally real, the fields $L_i$ must be real since they are fixed by complex conjugation. We have thus constructed infinitely many real Galois extensions $L_i/\Q$, $i \in \N$ such that $\Gal(L_i/\Q) \cong \Z/d\Z$. Moreover since for $i \neq j$, the primes $p_i$ and $p_j$ are different, we know that $K_i \cap K_j = \Q(\zeta_{p_i}) \cap \Q(\zeta_{p_j}) = \Q(\zeta_{\text{gcd}(p_i, p_j)}) = \Q$. By consequence also $L_i \cap L_j = \Q$. 
	\end{proof}
	
	\begin{theorem}
		\label{thm:numberofrational}
		The Lie algebras $\n^\R_{\G,c}$ and $\n^\C_{\G,c}$ associated to a simple undirected graph $\G$ have either exactly one or infinitely many rational forms up to $\Q$-isomorphism. The former is true if and only if $\Aut(\overline{\G})$ is trivial.
	\end{theorem}

	\begin{proof}
		If $\Aut(\overline{\G})$ is trivial, then clearly $H^1(\Qbar \cap \R/\Q, \Aut(\overline{\G}))$ is trivial as well which implies by the discussion above that both $\mathcal{F}_{\Q}(\n^\R)$ and $\mathcal{F}_{\Q}(\n^\C)$ count only one element. So from now on we assume that $\Aut(\overline{\G})$ is non-trivial and show that there infinitely many rational forms. 

		Since $\Aut(\overline{\G})$ is not trivial, there exists an element $\varphi \in \Aut(\overline{\G})$ of prime order $p$. Let $L_i$ with $i \in \N$ be the finite degree Galois extensions of $\Q$ with Galois group $\Z/p\Z$ as in Lemma \ref{lem:infinitelyManyPrimeGaloisExt}. Choose for all $i \in \N$ a generator $\sigma_i \in \Gal(L_i/\Q)$ and define the injective morphisms
		\[\rho_i':\Gal(L_i/\Q) \to \Aut(\overline{\G}): \sigma_i^k \mapsto \varphi^k. \]
		The fields $L_i$ are all different and hence the corresponding rational forms are non-isomorphic by Theorem \ref{thm:injectiveVersionClassificationRationalForms}.
		Since each $L_i$ is a real field, complex conjugation lies in the kernel of each $\rho_i$ and thus $\n_{\rho_i, c}^\Q$ is a rational form of $\n^\R_{\G,c}$ for all $i \in \N$. This proves that $\mathcal{F}_\Q(\n^\R_{\G,c})$ counts infinitely many elements. Because we have an injection $\mathcal{F}_\Q(\n^\R_{\G,c}) \to \mathcal{F}_\Q(\n_{\G,c}^\C):[\m^\Q] \mapsto [\m^\Q]$, this proves as well that $\mathcal{F}_\Q(\n^\C_{\G,c})$ counts infinitely many elements.
	\end{proof}
	
	As a consequence, we present a family of graphs such that the corresponding real and complex Lie algebras have a unique rational form, up to $\Q$-isomorphism.	
	
	\begin{example}
		Let $p, q$ be two non-negative integers with $q > 1$. Take two disjoint sets $S_1$ and $S_2$ which have cardinalities $p$ and $q$, respectively. We can define a simple undirected graph $\G = (S_1 \cup S_2, E)$ where $E = \{ \{\alpha, \beta\} \mid \alpha \in S_1, \beta \in S_2 \} \cup \{ \{\alpha, \beta\} \mid \alpha, \beta \in S_1, \alpha \neq \beta \}$. These type of graphs are called \textit{magnet graphs} and we say $S_1$ is the \textit{core} of $\G$. The quotient graph of $\G$ is equal to
		
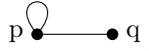
\[ 
		\begin{tikzpicture}[every loop/.style={}]
			\draw (0,0) -- (1,0);
			\filldraw [black] (0,0) circle (2pt) node[anchor = east] {p\,};
			\draw[scale = 1.5] (0,0)  to[in=50,out=130, loop] (0,0);
			\filldraw [black] (1,0) circle (2pt) node[anchor = west] {\,q};
		\end{tikzpicture}.
		\]
		From this it is clear that $\Aut(\overline{\G})$ is trivial and thus that the $c$-step nilpotent Lie algebras over $\R$ (or $\C$) which are associated to magnet graphs have only one rational form up to $\Q$-isomorphism. Note that these graphs were also considered in \cite{dm05-1} in the study of Anosov diffeomorphisms.
	\end{example}

	\begin{question}
		Does Theorem \ref{thm:numberofrational} hold for all real and complex (nilpotent) Lie algebras? The authors do not know any example of a (nilpotent) Lie algebra having at least two non-isomorphic rational forms, but only a finite number of them.
	\end{question}

	\subsection{Real forms of $\n_{\G, c}^\C$}
	\label{sec:realForms}
	
	In this section we prove Theorem \ref{thm:ClassificationOfRealForms} and apply it to some examples. Recall that if $G$ is a group, an element $g \in G$ is said to be an \textit{involution} of $G$ if $g^2 = 1$. Note that the neutral element $e \in G$ in particular is considered as an involution. 
	
	\begin{proof}[Proof of Theorem \ref{thm:ClassificationOfRealForms}]
		Combining Theorem \ref{thm:GaloisCohAutnG} for $K = \R$ and Theorem \ref{thm:GaloisDescentLieAlgebras} we get a bijection
		\[ H^1(\C/\R, \Aut(\overline{\G})) \to \mathcal{F}_\R(\n^\C_{\G, c}): [\rho] \mapsto [\n_{\rho, c}^\R]. \]
		Let $\tau \in \Gal(\C/\R)$ denote the complex conjugation automorphism. Then it is clear that any $\rho \in Z^1(\C/\R, \Aut(\overline{\G}))$ is determined by the involution $\rho(\tau)$. Conversely, every involution of $\varphi \in \Aut(\overline{\G})$ gives a unique morphism $\rho_\varphi:\Gal(\C/\R) \to \Aut(\overline{\G}): 1 \mapsto \Id,\, \tau \mapsto \varphi$. If we write $\n_{\varphi, c}^\R$ for the real form $\n_{\rho_\varphi, c}^\R$, then the statements in the theorem follow immediately.
	\end{proof}
	
	The following fact about symmetric groups is easily verified and thus we omit the proof.
	
	\begin{lemma}
		\label{lem:involutionsOfSn}
		For any $k \in \Z^{>0}$, the number of involutions up to conjugation in the symmetric group $\Perm(X)$ with $|X| = 2k$ or $|X| = 2k+1$ is equal to $k+1$.
	\end{lemma}
	
	In what follows, we will give examples of graphs for which we discuss the real forms of the associated complex Lie algebra.

	\begin{example}[$n$-fold direct sum of Heisenberg Lie algebras]
		Let $n > 1$ be an integer and let $\G_n = (X_n, E_n)$ denote the graph defined by $X_n = \{ 1, \ldots, 2n\}$ and $E_n = \{ \{2i-1, 2i\} \mid 1 \leq i \leq n \}$ as drawn below.
		\usetikzlibrary{patterns,snakes}
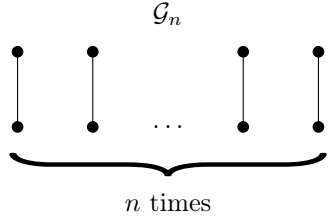
		\begin{figure}[H]
			\centering
			\begin{tikzpicture}
				\draw (-2, -1) -- (-2, 0);
				\draw (-1, -1) -- (-1, 0);
				\draw (1, -1) -- (1, 0);
				\draw (2, -1) -- (2, 0);
				
				\filldraw [black] (-2, 0) circle (2pt);
				\filldraw [black] (-1, 0) circle (2pt);
				\filldraw [black] (1, 0) circle (2pt);
				\filldraw [black] (2, 0) circle (2pt);
				
				\filldraw [black] (-2, -1) circle (2pt);
				\filldraw [black] (-1, -1) circle (2pt);
				\filldraw [black] (1, -1) circle (2pt);
				\filldraw [black] (2, -1) circle (2pt);
				
				\node at (0, -1) {$\ldots$};
				\node[cm={0,2.5,-12,0,(0,-1.5)}] at (0, 0) {$\{$};
				\node at (0, -2) {$n$ times};
				\node at (0, 0.5) {$\G_n$};
			\end{tikzpicture}
		\end{figure}
		If we let $\mathfrak{h}_3^K$ denote the 3-dimensional Heisenberg Lie algebra over the field $K$, then it is not hard to see that for $c = 2$ we get a decomposition
		\[ \n_{\G_n, 2}^\C \cong \underbrace{\mathfrak{h}_3^\C \oplus \ldots \oplus \mathfrak{h}_3^\C}_{n\text{ times}}.\]
		The set of coherent components of $\G_n$ is given by $\Lambda_n = \{\{2i-1, 2i\} \mid 1 \leq i \leq n\}$. It is straightforward to verify that the automorphism group $\Aut(\overline{\G_n})$ is isomorphic to the permutation group on a set with $n$ elements. By lemma \ref{lem:involutionsOfSn}, we thus find for any integers $c > 1$ and $k \geq 1$ that the number of real forms of $\n_{\G_n, c}^\C$ is equal to $k+1$ both for $n = 2k$ and $n = 2k+1$. In particular, we can achieve every natural number $ \geq 1$.
		
		In the special case of $c = 2$, we can even describe the Lie algebras explicitly. Let $\n_{-1, \ast}^\Q$ be the rational Lie algebra as defined in Example \ref{ex:twoCopiesHeisenbergComplement} and write $\n_{-1, \ast}^\R = \n_{-1, \ast}^\Q \otimes_\Q \R$. Note that $\n_{-1,\ast}^\R$ is a $6$-dimensional real Lie algebra that is also isomorphic to the underlying real Lie algebra of $\mathfrak{h}_3^\C$, so by restricting scalar multiplication on $\mathfrak{h}_3^\C$ to the real numbers. Every real form of $\n_{\G_n, 2}^\C$ is isomorphic to a direct sum Lie algebra of the form
		\[ \underbrace{\mathfrak{h}_3^\R \oplus \ldots \oplus \mathfrak{h}_3^\R}_{k\text{ times}} \: \oplus \: \underbrace{\mathfrak{n}_{-1, \ast}^\R \oplus \ldots \oplus \mathfrak{n}_{-1, \ast}^\R}_{l\text{ times}} \]
		for some non-negative integers $k, l$ which satisfy $k + 2l = n$.
	\end{example}

	In the example above, the real forms arise from permuting the summands in the direct sum decomposition. To show that any number of real forms can also be present in an indecomposable Lie algebra (see section \ref{sec:indecomposableForms}), we give the following example. 

	\begin{example}
		Let $n > 1$ be an integer and let $\mathcal{T}_n = (Y_n, F_n)$ denote the graph defined by $Y_n = \{ 1, \ldots, 2n + 1\}$ and $F_n = \{ \{1, \,2i\} \mid 1 \leq i \leq n \} \cup \{ \{2i, 2i + 1\} \mid 1 \leq i \leq n \}$ as drawn below.
		\usetikzlibrary{patterns,snakes}
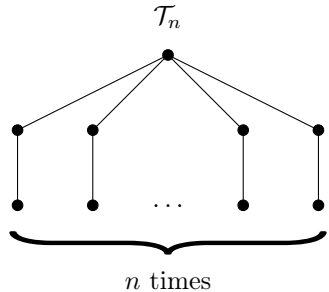
		\begin{figure}[H]
			\centering
			\begin{tikzpicture}
				\draw (0, 1) -- (-2, 0);
				\draw (0, 1) -- (-1, 0);
				\draw (0, 1) -- (1, 0);
				\draw (0, 1) -- (2, 0);
				
				\draw (-2, -1) -- (-2, 0);
				\draw (-1, -1) -- (-1, 0);
				\draw (1, -1) -- (1, 0);
				\draw (2, -1) -- (2, 0);
				
				\filldraw [black] (0, 1) circle (2pt);
				
				\filldraw [black] (-2, 0) circle (2pt);
				\filldraw [black] (-1, 0) circle (2pt);
				\filldraw [black] (1, 0) circle (2pt);
				\filldraw [black] (2, 0) circle (2pt);
				
				\filldraw [black] (-2, -1) circle (2pt);
				\filldraw [black] (-1, -1) circle (2pt);
				\filldraw [black] (1, -1) circle (2pt);
				\filldraw [black] (2, -1) circle (2pt);
				
				\node at (0, -1) {$\ldots$};
				\node[cm={0,2.5,-12,0,(0,-1.5)}] at (0, 0) {$\{$};
				\node at (0, -2) {$n$ times};
				\node at (0, 1.5) {$\mathcal{T}_n$};
			\end{tikzpicture}
		\end{figure}
		The set of coherent components is given by all singletons $\Lambda_n = \{\{i\} \mid 1 \leq i \leq 2n +1\}$. It is straightforward to verify that the automorphism group $\Aut(\overline{\mathcal{T}_n})$ is isomorphic to the permutation group on a set with $n$ elements. By lemma \ref{lem:involutionsOfSn}, we thus find for any integers $c > 1$ and $k \geq 1$ that the number of real forms of $\n_{\mathcal{T}_n, c}^\C$ is equal to $k+1$ both for $n = 2k$ and $n = 2k+1$. In particular, we can achieve every natural number $ \geq 1$.
	\end{example}

	\section{Indecomposable forms}
	\label{sec:indecomposableForms}

	Example \ref{ex:twoCopiesHeisenberg} illustrates that the direct sum of two complex Heisenberg Lie algebras has  rational forms that do not admit a direct sum of two non-trivial rational Lie algebras, in contrast to the original complex Lie algebra. We say those rational forms are indecomposable and in what follows we determine which forms of a Lie algebra associated to a graph have this property. First, let us give a rigorous definition of indecomposable Lie algebras.
	
	\begin{definition}
		A Lie algebra $\mathfrak{g}^K$ over a field $K$ is said to be \textit{decomposable} if there exist two non-trivial Lie ideals $\mathfrak{h}, \mathfrak{k} \subset \mathfrak{g}^K$ such that $\mathfrak{g}^K = \mathfrak{h} \oplus \mathfrak{k}$. We say a Lie algebra is \textit{indecomposable} if it is not decomposable.
	\end{definition}
	
	\noindent Note that if a Lie algebra $\g^L$ is indecomposable, then any $K$-form $\g^K$ for $K \subset L$ is indecomposable as well. As mentioned above, the converse does not hold. 
	
	If a Lie algebra is decomposable, one can decompose it into its indecomposable summands. Such a decomposition is not unique in general, but in \cite[Theorem 3.3.]{fgh13-1}, uniqueness was proven in case the Lie algebra is centreless, i.e. in case $Z(\g) = \{X \in \g \mid \forall \,Y \in \g:\, [X, Y] = 0 \} = \{0\}$. The theorem was proven for real Lie algebras, but the proof works for any subfield of $\C$. We can restate the result that we need as follows:
	
	\begin{theorem}[Fisher, Gray, Hydon]
		\label{thm:uniquenessDecompositionLieAlgebras}
		Let $\g$ be a Lie algebra over a field $K \subset \C$, $s, t$ positive integers and $\mathfrak{h}_1, \ldots, \mathfrak{h}_s, \mathfrak{k}_1, \ldots, \mathfrak{k}_t \subset \g$ indecomposable ideals such that
		\[ \g = \mathfrak{h}_1  \oplus  \ldots \oplus \mathfrak{h}_s \quad \quad \text{and} \quad \quad \g = \mathfrak{k}_1 \oplus \ldots \oplus \mathfrak{k}_t,\]
		then $s = t$ and up to reordering the summands $\mathfrak{h}_i$, we have $\mathfrak{h}_i \subset \mathfrak{k}_i + Z(\g)$ for all $i \in \{1, \ldots, s\}$.
	\end{theorem}
	
	We want to apply the above theorem to Lie algebras associated to graphs. In order to do so we first need to determine some decomposition of $\n_{\G, c}^K$ into decomposable summands. For any two Lie algebras $\g_1, \g_2$, the \textit{direct sum Lie algebra} is the vector space $\g_1 \oplus \g_2$ endowed with a Lie bracket defined by $[X_1 + X_2, Y_1 + Y_2] := [X_1, Y_1] + [X_2, Y_2]$ for any $X_1, Y_1 \in \g_1$ and $X_2, Y_2 \in \g_2$.
	
	Note that the Lie algebra $\n_{\G, c}^K$ satisfies the universal property that for any other $c$-step nilpotent Lie algebra $\g$ over $K$ and a map $i:S \to \g$ which satisfies $[i(\alpha), i(\beta)] = 0$ for any $\alpha, \beta \in S$ with $\{\alpha, \beta\} \notin E$, there exists a unique Lie algebra morphism $f:\n_{\G, c}^K \to \g$ such that $f$ restricts to $i$ on $S$. Using this property, it is not hard to show the following:
	
	\begin{lemma}
		\label{lem:decompositionLieAlgGraphs}
		Let $\G = (S, E)$ be a graph and $K \subset \C$ a field. Let $S_1, S_2 \subset S$ be disjoint sets such that $S = S_1 \cup S_2$ and $\forall \alpha \in S_1, \beta \in S_2: \, \{\alpha , \beta\} \notin E$. Write $\G_1, \, \G_2$ for the subgraphs of $\G$ spanned by $S_1, S_2$, respectively. There exists a unique Lie algebra isomorphism
		\[ \n_{\G, c}^K \stackrel{\cong}{\longrightarrow} \n_{\G_1, c}^K \oplus \n_{\G_2, c}^K \]
		which restricts to the identity on $S$.
	\end{lemma}	
	
	Recall that for a graph $\G = (S, E)$, we say two vertices $\alpha, \beta \in S$ are \textit{connected} if there exists a positive integer $k$ and vertices $\gamma_1 = \alpha, \gamma_2, \ldots, \gamma_{k-1}, \gamma_k = \beta \in S$ such that $\{\gamma_i, \gamma_{i+1}\} \in E$ for all $i \in \{1, \ldots, k-1\}$. We say the graph $\G$ is \textit{connected} if all pairs of vertices in $S$ are connected.
	
	\begin{lemma}
		\label{lem:indecomposableLieAlgGraph}
		Let $\G = (S, E)$ be a graph and $K \subset \C$ a field. The following are equivalent:
		\begin{enumerate}[label = (\roman*)]
			\item \label{item:TFAE1} $\G$ is connected,
			\item \label{item:TFAE2} $\n^K_{\G, 2}$ is indecomposable,
			\item \label{item:TFAE3} $\n^K_{\G, c}$ is indecomposable for any $c > 1$.
		\end{enumerate}
	\end{lemma}

	\begin{proof}
		\ref{item:TFAE1} $\Rightarrow$ \ref{item:TFAE2}. The rational case follows from the `conversely' part of the proof of Proposition 4.2. in \cite{send22-1}. By taking a field extension, it holds for any field $K \subset \C$.
		
		\ref{item:TFAE2} $\Rightarrow$ \ref{item:TFAE3}. Take any $c > 1$ and assume that there exist ideals $\mathfrak{h}, \mathfrak{k} \subset \n_{\G, c}^K$ such that $\n_{\G, c}^K = \mathfrak{h} \oplus \mathfrak{k}$. Note that both $\mathfrak{h}$ and $\mathfrak{k}$ are also nilpotent Lie algebras. It follows that we have a sequence of isomorphisms:
		\[ \n_{\G, 2}^K \cong \n_{\G, c}^K/\gamma_{3}(\n_{\G, c}^K) \cong \mathfrak{h}/\gamma_3(\mathfrak{h}) \oplus \mathfrak{k}/\gamma_3(\mathfrak{k}). \]
		Since $\n_{\G, 2}^K$ is assumed to be indecomposable, this implies that $\mathfrak{h}/\gamma_3(\mathfrak{h}) = \{0\}$ or $\mathfrak{k}/\gamma_3(\mathfrak{k}) = \{0\}$. By nilpotency of $\mathfrak{h}$ and $\mathfrak{k}$, we then get that $\mathfrak{h} = \{0\}$ or $\mathfrak{k} = \{0\}$.
		
		\ref{item:TFAE3} $\Rightarrow$ \ref{item:TFAE1}. We prove this by contraposition. Assume $\G$ is not connected, then there exist a partition of the vertices $S = S_1 \cup S_2$ with $S_1$ and $S_2$ non-empty, such that for any $\alpha \in S_1$ and any $\beta \in S_2$ it holds that $\{\alpha, \beta\} \notin E$. It then readily follows from Lemma \ref{lem:decompositionLieAlgGraphs} that $\n_{\G, c}^K$ is decomposable and concludes the proof.
	\end{proof}

	Note that, as mentioned in section \ref{sec:formsNilpLieAlgGraphs}, there is a canonical vector space isomorphism $V \cong \n_{\G, c}^K/[\n_{\G, c}^K, \n_{\G, c}^K]: v \mapsto v + [\n_{\G, c}^K, \n_{\G, c}^K]$. Let us write
	\[\pi_{\ab}:\n_{\G,c}^K \to V\]
	for the projection onto the abelianization. The projection of the centre to the abelianization can be described by the following lemma. The \textit{degree} of a vertex $\alpha \in S$ is defined as the number of vertices adjacent to $\alpha$.
	
	\begin{lemma}
		\label{lem:projCentreAb}
		For any graph $\G = (S, E)$ and field $K \subset \C$ we have
		\[ \pi_{\ab}\left(Z\left(\n_{\G, c}^K\right)\right) = \Sp_K\left(\{ \alpha \in S \mid \alpha \text{ has degree } 0 \}\right). \]
	\end{lemma}

	\begin{proof}
		Take any $\alpha \in S$ of degree 0. It is clear that $[\alpha, \beta] = 0$ in $\n_{\G, c}^K$ for any $\beta \in S$ and thus since $S$ generates $\n_{\G, c}^K$ that $\alpha \in Z(\n_{\G, c}^K)$. By consequence we get the inclusion $$\Sp_K\left(\{ \alpha \in S \mid \alpha \text{ has degree } 0 \}\right) \subseteq \pi_{\ab}\left(Z\left(\n_{\G, c}^K\right)\right).$$
		
		For the other inclusion, take any $v \in \pi_{\ab}\left(Z\left(\n_{\G, c}^K\right)\right)$. Then there exists a $w \in [\n_{\G, c}^K, \n_{\G, c}^K]$ such that $v + w \in Z\left(\n_{\G, c}^K\right)$. By consequence, for any $\beta \in S$ it must hold that $[\beta, v+w] = [\beta, v] + [\beta, w] = 0$. From the grading of $\n_{\G, c}^K$ as given in (\ref{eq:gradingPartComLieAlg}) it follows that $[V, V] \cap [V, [\n_{\G, c}^K, \n_{\G, c}^K]] = \{0\}$ and thus that $\forall \beta \in S$ it holds that $[\beta, v] = 0$. Let $f:S \to K$ be the unique function such that $v = \displaystyle \sum_{\alpha \in S} f(\alpha) \alpha$. We then get that $\displaystyle \sum_{\alpha \in S} f(\alpha)[\beta, \alpha] = 0$ for any $\beta \in S$. Using the relations in $\n_{\G, c}^K$, this reduces to $\displaystyle \sum_{\alpha \in \Omega'(\beta)} f(\alpha)[\beta, \alpha] = 0$, where we remind the reader of the definition of $\Omega'(\beta)$ in (\ref{eq:openAndClosedNeigh}). Since the set $\{[\beta, \alpha] \mid \alpha \in \Omega'(\beta)\}$ is linearly independent in $\n_{\G, c}^K$, it follows that $f(\alpha) = 0$ for all $\alpha \in \Omega'(\beta)$ and all $\beta \in S$. This exactly means that $v \in \Sp_K\left(\{ \alpha \in S \mid \alpha \text{ has degree } 0 \}\right)$, which proves the other inclusion.
	\end{proof}
	
	The relation `being connected' as defined above is an equivalence relation on $S$. The equivalence classes are called the \textit{connected components} of $\G$. We let $\mathcal{C}(\G)$ denote the set of all connected components of $\G$. We can now combine Theorem \ref{thm:uniquenessDecompositionLieAlgebras}, Lemma \ref{lem:decompositionLieAlgGraphs}, Lemma \ref{lem:indecomposableLieAlgGraph} and Lemma \ref{lem:projCentreAb} to prove the following result for the decomposition of Lie algebras associated to graphs.
	
	\begin{prop}
		\label{prop:uniquenessDecompositionLieAlgGraphs}
		Let $\G = (S, E)$ be a graph with no vertices of degree 0, $K \subset \C$ a field $k$ a positive integer and $\mathfrak{h}_1, \ldots, \mathfrak{h}_k \subset \n_{\G, c}^K$ ideals such that $\n_{\G, c}^K = \mathfrak{h}_1 \oplus \ldots \oplus \mathfrak{h}_k$. Then $k = |\mathcal{C}(\G)|$ and there exists an ordering of the connected components $\mathcal{C}(\G) = \{C_1, \ldots, C_k\}$ such that $\pi_{\ab}(\mathfrak{h}_i) = \Sp_K(C_i)$ for all $i \in \{1, \ldots, k\}$.
	\end{prop}

	\begin{proof}
		Let $\{C_1, \ldots, C_l\}$ be some ordering of the connected components of $\G$ with $l = |\mathcal{C}(\G)|$. Let $\n_i \subset \n_{\G, c}^K$ be the Lie subalgebra in $\n_{\G, c}^K$ generated by $C_i$ for all $i \in \{1, \ldots, l\}$. By Lemma \ref{lem:decompositionLieAlgGraphs}, it is clear that the $\n_i$'s are ideals, that $\n_{\G, c}^K = \n_1 \oplus \ldots \oplus \n_k$ and if $\G_i$ denotes the subgraph spanned by $C_i$, that $\n_i \cong \n_{\G_i, c}^K$. Moreover, since each graph $\G_i$ is connected, Lemma \ref{lem:indecomposableLieAlgGraph} implies that $\n_i \cong \n_{\G_i , c}^K$ is indecomposable for each $i \in \{1, \ldots, l\}$. By Theorem \ref{thm:uniquenessDecompositionLieAlgebras}, we have $k = l = |\mathcal{C}(\G)|$ and we can fix a reordering of the connected components of $\G$ such that $\mathfrak{h_i} \subset \n_i + Z(\n_{\G, c}^K)$ for all $i \in \{1, \ldots , k\}$. Since $\G$ has no vertices of degree 0, we find by Lemma \ref{lem:projCentreAb} that $\pi_{\ab}(Z(\n_{\G, c}^K)) = \{0\}$ and thus that $\pi_{\ab}(\mathfrak{h}_i) \subset \pi_{\ab}(\n_i)$ for all $i \in \{1, \ldots, k\}$. Note that $\pi_{ab}(\n_i) = \Sp_K(C_i)$ and thus that $V = \pi_{ab}(\n_1) \oplus \ldots \oplus \pi_{\ab}(\n_k)$. Since $\pi_{\ab}$ is surjective we must have $V = \pi_{\ab}(\mathfrak{h}_1 + \ldots + \mathfrak{h}_k) = \pi_{\ab}(\mathfrak{h}_1) + \ldots + \pi_{\ab}(\mathfrak{h}_k)$ which implies that $\pi_{\ab}(\mathfrak{h}_i) = \pi_{ab}(\n_i) = \Sp_K(C_i)$ for all $i \in \{1, \ldots, k\}$. This concludes the proof.
 	\end{proof}
	
	Let $\G = (S, E)$ be a graph and $L/K$ a Galois extension of subfields of $\C$. Recall the natural action of $\Gal(L/K)$ on $\n_{\G, c}^L$. Note that $\Sp_L(S) = V \subset \n_{\G, c}^L$ is preserved under this action and thus that we have an induced action of $\Gal(L/K)$ on $V$. The vertices $S \subset V$ are fixed under this action. If $\rho:\Gal(L/K) \to \Aut(\G)$ is a continuous morphism and $\n_{\rho, c}^K$ the associated $K$-form of $\n_{\G, c}^L$, one can check that
	\begin{equation}
		\label{eq:projRatFormOntoAbelianization}
		\pi_{\ab}(\n_{\rho, c}^K) = \{v \in V \mid \forall \sigma \in \Gal(L/K): \overline{P}(\rho_\sigma)(\asi{v}) = v\}
	\end{equation}
	where $\overline{P}:\Aut(\overline{\G}) \to \GL(V)$ is the morphism as defined in section \ref{sec:quotientGraph}.
	
	Let $\Lambda$ denote the set of coherent components of $\G$. As one can check, every connected component which counts at least two vertices is a disjoint union of coherent components. On the other hand, the union of all connected components which are singletons is equal to a coherent component. This is illustrated by the example drawn below. Dashed lines represent coherent components while full lines (apart from de edges) represent connected components.
	
	
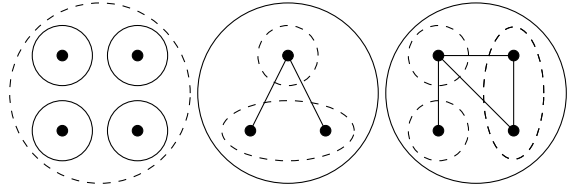
\begin{figure}[H]
		\centering
		\begin{tikzpicture}
			\filldraw [black] (-0.5, -0.5) circle (2pt);
			\filldraw [black] (0.5, -0.5) circle (2pt);
			\filldraw [black] (-2, -0.5) circle (2pt);
			\filldraw [black] (-3, -0.5) circle (2pt);
			\filldraw [black] (2, -0.5) circle (2pt);
			\filldraw [black] (3, -0.5) circle (2pt);
			\filldraw [black] (0, 0.5) circle (2pt);
			\filldraw [black] (-2, 0.5) circle (2pt);
			\filldraw [black] (-3, 0.5) circle (2pt);
			\filldraw [black] (2, 0.5) circle (2pt);
			\filldraw [black] (3, 0.5) circle (2pt);
			
			\draw (0, 0.5) -- (-0.5, -0.5);
			\draw (0, 0.5) -- (0.5, -0.5);
			\draw (2, 0.5) -- (3, 0.5);
			\draw (2, 0.5) -- (2, -0.5);
			\draw (3, 0.5) -- (3, -0.5);
			\draw (3, -0.5) -- (2, 0.5);
			
			\draw (0,0) circle (1.2);
			\draw (2.5,0) circle (1.2);
			\draw[dashed] (-2.5,0) circle (1.2);
			
			\draw (-3,0.5) circle (0.4);
			\draw (-2,0.5) circle (0.4);
			\draw (-3,-0.5) circle (0.4);
			\draw (-2,-0.5) circle (0.4);
			
			\draw[dashed] (0,0.5) circle (0.4);
			
			\draw[dashed] (2,0.5) circle (0.4);
			\draw[dashed] (2,-0.5) circle (0.4);
			
			\draw[dashed, cm={1,0,0,2.2,(0,0)}] (3, 0) circle (0.4);
			\draw[dashed, cm={1,0,0,2.2,(0,0)}] (3, 0) circle (0.4);
			\draw[dashed, cm={2.2,0,0,1,(0,0)}] (0, -0.5) circle (0.4);
		\end{tikzpicture}
	\end{figure}
	
	For any $\varphi \in \Aut(\overline{\G})$, the images of two coherent components which are subsets of the same connected component are again subsets of the same connected component. Let $C \in \mathcal{C}(\G)$ be a connected component. In case $C$ counts at least two vertices, there exist coherent components $\lambda_1, \ldots, \lambda_k$ such that $C = \lambda_1 \sqcup \ldots \sqcup \lambda_k$ and we define $\chi(\varphi)(C) = \varphi(\lambda_1) \sqcup \ldots \sqcup \varphi(\lambda_k)$. In case $C$ is a singleton, we define $\chi(\varphi)(C) = C$. This gives a morphism
	\[ \chi:\Aut(\overline{\G}) \to \Perm(\mathcal{C}(\G)). \]
	Recall that an action is called \textit{transitive} if it has only one orbit.
	
	\begin{theorem}
		Let $\G = (S, E)$ be a graph, $L/K$ a Galois extension of subfields of $\C$ and $\rho:\Gal(L/K) \to \Aut(\overline{\G})$ a continuous morphism. The Lie algebra $\n_{\rho, c}^K$ is indecomposable if and only if the $\chi \circ \rho$-action on the set of connected components $\mathcal{C}(\G)$ is transitive.
	\end{theorem}

	\begin{proof}
		Assume $\n_{\rho, c}^K$ is indecomposable. Let $C \in \mathcal{C}(\G)$ be a connected component. We write $\tilde{S}$ for the vertices which lie in a connected component which lies in the $\chi \circ \rho$-orbit of $C$, i.e. $\tilde{S} = \{ \alpha \in S \mid \exists \sigma \in \Gal(L/K): \alpha \in \chi(\rho_\sigma)(C) \}$. Let $\n_1$ and $\n_2$ denote the subalgebras of $\n_{\G, c}^L$ generated by $\tilde{S}$ and $S \setminus \tilde{S}$, respectively. It follows that $r(\rho_\sigma)(\tilde{S}) = \tilde{S}$ and by consequence that $i(\rho_\sigma)(\n_1) = \n_1$ and $i(\rho_\sigma)(\n_2) = \n_2$ for any $\sigma \in \Gal(L/K)$. Moreover, since $\n_1$ and $\n_2$ are generated by sets of vertices, we have that $i(\rho_\sigma)(\asi{\n_1}) = \n_1$ and $i(\rho_\sigma)(\asi{\n_2}) = \n_2$ for any $\sigma \in \Gal(L/K)$. By Lemma \ref{lem:decompositionLieAlgGraphs}, $\n_1$ and $\n_2$ are ideals and we have a direct sum $\n_{\G, c}^L = \n_1 \oplus \n_2$. Thus, for an arbitrary $v \in \n_{\G, c}^L$ there exist unique vectors $v_1 \in \n_1$ and $v_2 \in \n_2$ such that $v = v_1 + v_2$. For any $\sigma \in \Gal(L/K)$, we then have the equivalences
		\begin{align*}
			i(\rho_\sigma)(\asi{v}) = v &\Leftrightarrow i(\rho_\sigma)(\asi{(v_1 + v_2)}) = v_1 + v_2\\
			&\Leftrightarrow i(\rho_\sigma)(\asi{v_1}) + i(\rho_\sigma)(\asi{v_2}) = v_1 + v_2\\
			&\Leftrightarrow i(\rho_\sigma)(\asi{v_1}) = v_1 \quad \text{and} \quad i(\rho_\sigma)(\asi{v_2}) = v_2.
		\end{align*}
		where the last equivalence uses $i(\rho_\sigma)(\asi{\n_j}) = \n_j$ for $j = 1,2$. Since $\n_1$, $\n_2$ are themselves Lie algebras associated to a graph, we know that
		\[ \mathfrak{m}_j = \{ v \in \n_j \mid \forall \sigma \in \Gal(L/K): i(\rho_\sigma)(v) = v \} \]
		defines a $K$-form of $\n_j$. The above equivalences then imply that $\n_{\rho, c}^K = \mathfrak{m}_1 \oplus \mathfrak{m}_2$. Moreover, since $\n_1$ and $\n_2$ are ideals, the same holds for $\mathfrak{m}_1$ and $\mathfrak{m}_2$. Since we assumed $\n_{\rho, c}^K$ to be indecomposable and $\mathfrak{m}_1 \neq \{0\}$ by construction, we must have that $\mathfrak{m}_2 = \{0\}$ and thus that $S \setminus \tilde{S} = \emptyset$. This proves that the $\chi \circ \rho$-action on $\mathcal{C}(\G)$ is transitive.
		
		Conversely, assume that the $\chi \circ \rho$-action on $\mathcal{C}(\G)$ is transitive. Since this action fixes connected components which are singletons, transitivity implies that either there are no connected components which are singletons or that $S$ itself is a singleton. In the latter case, the Lie algebra $\n_{\G, c}^L$ is itself indecomposable and by consequence so are all of its forms. Thus we can assume that there are no connected components which are singletons. This is equivalent to saying there are no vertices of degree 0. Assume there are ideals $\mathfrak{m}_1$, $\mathfrak{m}_2$ of the Lie algebra $\n_{\rho, c}^K$ such that $\n_{\rho, c}^K = \mathfrak{m}_1 \oplus \mathfrak{m}_2$. Define $\n_{j} = \Sp_L(\mathfrak{m}_j)$ for $j = 1, 2$. Since $\n_{\G, c}^K$ is a form of $\n_{\G, c}^L$, it follows that $\n_1$, $\n_2$ are ideals of $\n_{\G, c}^L$ and that $\n_{\G, c}^L = \n_1 \oplus \n_2$. By Proposition \ref{prop:uniquenessDecompositionLieAlgGraphs}, we must have subsets of vertices $S_1, S_2 \subset S$, each a union of connected components, such that $S = S_1 \sqcup S_2$ and $\pi_{\ab}(\n_j) = \Sp_L(S_j)$ for $j = 1, 2$. Take $\sigma \in \Gal(L/K)$ arbitrarily. Note that $\asi{(\Sp_L(S_j))} = \Sp_L(S_j)$ for $j = 1, 2$ since these subspaces are spanned by vertices. Using that $\pi_{\ab}(\mathfrak{m}_j) \subset \Sp_L(S_j)$, we thus find that $\asi{(\pi_{\ab}(\m_j))} \subset \Sp_L(S_j)$ for $j = 1, 2$. From equation (\ref{eq:projRatFormOntoAbelianization}) it follows that $\overline{P}(\rho_\sigma)^{-1}(\pi_{\ab}(\mathfrak{m}_j)) = \prescript{\sigma}{}{\mathfrak{m}_j} \subset{\Sp_L(S_j)}$. But since $\Sp_L(\mathfrak{m}_j) = \Sp_L(S_j)$, we also have $\overline{P}(\rho_\sigma)^{-1}(\Sp_L(S_j)) = \Sp_L(S_j)$. Since $\sigma$ was chosen arbitrarily, this implies that $r(\rho_\sigma)(S_j) = S_j$ for all $\sigma \in \Gal(L/K)$. Since we assumed that the $\chi \circ \rho$-action on $\mathcal{C}(\G)$ is transitive, it follows that either $S_1$ or $S_2$ is the empty set and thus that either $\mathfrak{m}_1 = \{0\}$ or $\mathfrak{m}_2 = \{0\}$.
	\end{proof}

	\bibliography{ref}

\begin{thebibliography}{10}

\bibitem{berh10-1}
Gr\'{e}gory Berhuy.
\newblock {\em An introduction to {G}alois cohomology and its applications},
  volume 377 of {\em London Mathematical Society Lecture Note Series}.
\newblock Cambridge University Press, Cambridge, 2010.
\newblock With a foreword by Jean-Pierre Tignol.

\bibitem{dm05-1}
S.~G. Dani and Meera~G. Mainkar.
\newblock Anosov automorphisms on compact nilmanifolds associated with graphs.
\newblock {\em Trans. Amer. Math. Soc.}, 357(6):2235--2251, 2005.

\bibitem{ddm18-1}
Rachelle~C. DeCoste, Lisa DeMeyer, and Meera~G. Mainkar.
\newblock Graphs and metric 2-step nilpotent {L}ie algebras.
\newblock {\em Adv. Geom.}, 18(3):265--284, 2018.

\bibitem{dd14-1}
Karel Dekimpe and Jonas Der\'{e}.
\newblock Expanding maps and non-trivial self-covers on infra-nilmanifolds.
\newblock {\em Topol. Methods Nonlinear Anal.}, 47(1):347--368, 2016.

\bibitem{dv11-1}
Karel Dekimpe and Kelly Verheyen.
\newblock Constructing infra-nilmanifolds admitting an {A}nosov diffeomorphism.
\newblock {\em Adv. Math.}, 228(6):3300--3319, 2011.

\bibitem{dere19-1}
Jonas Der\'{e}.
\newblock The structure of underlying {L}ie algebras.
\newblock {\em Linear Algebra Appl.}, 581:471--495, 2019.

\bibitem{dm21-1}
Jonas Deré and Meera Mainkar.
\newblock Anosov diffeomorphisms on infra-nilmanifolds associated to graphs.
\newblock {\em Math. Nachr.}, 2021.

\bibitem{dlv12-1}
Antonio~J. Di~Scala, Jorge Lauret, and Luigi Vezzoni.
\newblock Quasi-{K}\"ahler {C}hern-flat manifolds and complex 2-step nilpotent
  {L}ie algebras.
\newblock {\em Ann. Sc. Norm. Super. Pisa Cl. Sci. (5)}, 11(1):41--60, 2012.

\bibitem{diri13-1}
Peter Gustav~Lejeune Dirichlet.
\newblock {\em Beweis des Satzes, dass jede unbegrenzte arithmetische
  Progression, deren erstes Glied und Differenz ganze Zahlen ohne
  gemeinschaftlichen Factor sind, unendlich viele Primzahlen enthält}, page
  342–359.
\newblock Cambridge Library Collection - Mathematics. Cambridge University
  Press, 2013.

\bibitem{fgh13-1}
David~J. Fisher, Robert~J. Gray, and Peter~E. Hydon.
\newblock Automorphisms of real {L}ie algebras of dimension five or less.
\newblock {\em J. Phys. A}, 46(22):225204, 18, 2013.

\bibitem{gs84-1}
Fritz Grunewald and Dan Segal.
\newblock Reflections on the classification of torsion-free nilpotent groups.
\newblock In {\em Group theory}, pages 121--158. Academic Press, London, 1984.

\bibitem{helg01-1}
Sigurdur Helgason.
\newblock {\em Differential geometry, {L}ie groups, and symmetric spaces},
  volume~34 of {\em Graduate Studies in Mathematics}.
\newblock American Mathematical Society, Providence, RI, 2001.
\newblock Corrected reprint of the 1978 original.

\bibitem{laur08-1}
Jorge Lauret.
\newblock Rational forms of nilpotent {L}ie algebras and {A}nosov
  diffeomorphisms.
\newblock {\em Monatsh. Math.}, 155(1):15--30, 2008.

\bibitem{lw09-1}
Jorge Lauret and Cynthia~E. Will.
\newblock Nilmanifolds of dimension {$\leq 8$} admitting {A}nosov
  diffeomorphisms.
\newblock {\em Trans. Amer. Math. Soc.}, 361(5):2377--2395, 2009.

\bibitem{maplsa18-1}
Meera Mainkar, Matthew Plante, and Ben Salisbury.
\newblock Counting {A}nosov graphs.
\newblock {\em Ars Combin.}, 141:29--51, 2018.

\bibitem{main06-1}
Meera~G. Mainkar.
\newblock Anosov automorphisms on certain classes of nilmanifolds.
\newblock {\em Glasg. Math. J.}, 48(1):161--170, 2006.

\bibitem{malc49-2}
Anatoly~I. Maltsev.
\newblock Nilpotent torsion--free groups.
\newblock {\em Izvestiya Aka. Nauk SSSR., Ser. Mat.}, (13):pp. 201--212, 1949.

\bibitem{niko20-1}
Y.~Nikolayevsky.
\newblock Geodesic orbit and naturally reductive nilmanifolds associated with
  graphs.
\newblock {\em Math. Nachr.}, 293(4):754--760, 2020.

\bibitem{ovan22-1}
Gabriela~P. Ovando.
\newblock The geodesic flow on nilpotent {L}ie groups of steps two and three.
\newblock {\em Discrete Contin. Dyn. Syst.}, 42(1):327--352, 2022.

\bibitem{seme02-1}
Yu.~S. Semenov.
\newblock On the rational forms of nilpotent {L}ie algebras and lattices in
  nilpotent {L}ie groups.
\newblock {\em Enseign. Math. (2)}, 48(3-4):191--207, 2002.

\bibitem{send22-1}
Pieter Senden.
\newblock The reidemeister spectrum of direct products of nilpotent groups,
  2022.

\bibitem{serr02-1}
Jean-Pierre Serre.
\newblock {\em Galois cohomology}.
\newblock Springer Monographs in Mathematics. Springer-Verlag, Berlin, english
  edition, 2002.
\newblock Translated from the French by Patrick Ion and revised by the author.

\bibitem{smal67-1}
S.~Smale.
\newblock Differentiable dynamical systems.
\newblock {\em Bull. Amer. Math. Soc.}, 73,:pp. 747--817, 1967.

\bibitem{sulc21-1}
Diego Sulca.
\newblock A remark on the degree of polynomial subgroup growth of nilpotent
  groups.
\newblock preprint.

\end{thebibliography}
	\bibliographystyle{plain}	

\end{document}